\documentclass[12pt]{amsart}
\usepackage{graphicx,amssymb,latexsym,bm,color}
\usepackage[dvips,centering,includehead,width=15.1cm,height=22.7cm]{geometry}
\usepackage{amsmath}
\usepackage{graphicx}
\usepackage{epsfig}
\usepackage{amssymb}

\input{xy} \xyoption{all}

\numberwithin{equation}{section}

\newtheorem{theorem}{Theorem}[section]
\newtheorem{lemma}[theorem]{Lemma}

\newtheorem{corollary}[theorem]{Corollary}
\newtheorem{proposition}[theorem]{Proposition}

\theoremstyle{definition} \newtheorem{definition}[theorem]{Definition}
\newtheorem{example}[theorem]{Example}
\newtheorem{remark}[theorem]{Remark}
\newtheorem{notation}[theorem]{Notation}

\newcommand{\eps}{{\varepsilon}}
\newcommand{\M}{{\mathcal M}} 
 
\renewcommand{\L}{{\mathcal L}}

  \newcommand{\A}{{\mathcal A}} 
 
\newcommand{\D}{{\mathcal D}}

\newcommand{\dive}{{\text{div}}}
\newcommand{\inte}{{\text{int}}}
\newcommand{\Tor}{{\mathrm{S}}}

\newcommand{\conv}{{\text{conv}}}
\newcommand{\inv}{{\text{inv}}}
\newcommand{\Rev}{{\text{Rev}}}

 \newcommand{\RR}{{\mathbb R}}
\newcommand{\ZZ}{{\mathbb Z}} \newcommand{\NN}{{\mathbb N}}
 \newcommand{\PP}{{\mathbb C}{\mathbb P}}
\newcommand{\CC}{{\mathbb C}}

\newcommand{\oo}{\multiput(0,0)(10,0){2}{\circle*{2}}}
\newcommand{\ooo}{\multiput(0,0)(10,0){3}{\circle*{2}}}
\newcommand{\oooo}{\multiput(0,0)(10,0){4}{\circle*{2}}}

\newcommand{\Eeee}{\put(1,0){\line(1,0){8}}}
\newcommand{\eEee}{\put(11,0){\line(1,0){8}}}

\renewcommand{\1}{ {\bf 1} }
\newcommand{\m}{ {\bf -} }

\newcommand{\ee}{ {\bf e} }
\renewcommand{\aa}{ {\bf a} }
\newcommand{\dd}{ {\bf d} }
\newcommand{\ii}{ {\bf i} }

\renewcommand{\ll}{ {\bf l} }
\newcommand{\rr}{ {\bf r} }
\renewcommand{\ss}{ {\bf s} }
\newcommand{\kk}{ {\bf k} }

\newcommand{\yy}{ {\bf y} }
\newcommand{\cc}{ {\bf c} }
\newcommand{\GGamma}{ {\bf \Gamma} }

\title[Universal Polynomials for Severi Degrees of Toric Surfaces]{Universal Polynomials for \\ Severi Degrees of Toric Surfaces}

\author{Federico Ardila \and Florian Block}
\date{\today}
\address{Federico Ardila, Department of Mathematics, San Francisco
  State University, USA.}
\email{federico@sfsu.edu}
\address{Florian Block, Mathematics Institute, University of Warwick,
United Kingdom.}
\email{f.s.block@warwick.ac.uk}
\thanks {\emph {2010 Mathematics Subject Classification: Primary:
    14N10. Secondary: 51M20, 14N35. 
}}
\keywords {Enumerative geometry, toric surfaces, Gromov-Witten theory, Severi degrees, node polynomials}
\thanks{The first author was partially supported by the National Science Foundation CAREER Award DMS-0956178 and the National Science Foundation Grant DMS-0801075.
The second author was partially supported by the National Science Foundation Grant DMS-055588
and a Rackham Fellowship.}

\begin{document}

\begin{abstract}
The Severi variety parameterizes plane curves of degree~$d$ with
$\delta$~nodes. Its degree is called the Severi degree. For large
enough $d$, the Severi degrees coincide with the Gromov-Witten invariants
of $\PP^2$. Fomin
and Mikhalkin (2009) proved the 1995 conjecture that for
fixed~$\delta$, Severi degrees are eventually polynomial in $d$. 

In this paper, we study the Severi varieties corresponding to a large 
family of
toric surfaces. We prove the analogous result that the Severi degrees
are eventually polynomial as a function of the multidegree. More
surprisingly, we show that the Severi degrees are also eventually
polynomial ``as a function of the surface". We illustrate our theorems
by explicit computing, for a small number of nodes, the Severi degree
of any large enough Hirzebruch surface and of a singular surface.

Our strategy is to use tropical geometry to express Severi degrees in terms of Brugall\'e and Mikhalkin's floor diagrams, and study those combinatorial objects in detail. An important ingredient in the proof is the polynomiality of the discrete volume of a variable facet-unimodular polytope.
\end{abstract}

\maketitle

\section{Introduction and Main Theorems}
\label{sec:intro}

\subsection{Severi degrees and node polynomials for $\PP^2$.}
A \emph{$\delta$-nodal curve} is a reduced (not necessarily irreducible) curve having $\delta$ simple nodes and no other singularities. The \emph{Severi degree} $N^{d, \delta}$ is the degree of the Severi
variety parameterizing degree $d$ $\delta$-nodal curves in
the complex projective plane $\PP^2$. 
 In other
words, $N^{d, \delta}$ is the number of such curves through an appropriate
number of points in general position. For $d \geq \delta+2$, $N^{d, \delta}$ equals the Gromov-Witten invariant $N_{d, {d-1 \choose 2}-\delta}$. 


Severi varieties were introduced around 1915 by 
Enriques~\cite{En12} and Severi~\cite{Se21}, and have 
received considerable attention.
Much later,
in 1986, Harris~\cite{Ha86} achieved a celebrated breakthrough by
proving their irreducibility.

In 2009, Fomin and Mikhalkin \cite[Theorem 5.1]{FM} proved Di~Francesco and Itzykson's 1995 conjecture \cite{DI}
that, for a fixed number of
nodes $\delta$, the Severi degree $N^{d, \delta}$ becomes a polynomial
$N_{\delta}(d)$ in the degree, for $d \ge 2 \delta$. We will call $N_{\delta}(d)$ the
\emph{node polynomial} following Kleiman--Piene \cite{KP}. In
\cite{FB}, the second author improved the threshold of
Fomin and Mikhalkin from $2 \delta$ to $\delta$ and computed the node
polynomials for $\delta \le 14$ extending work of Kleiman and
Piene~\cite{KP} for $\delta \le 8$.

\subsection{Severi degrees and node polynomials for toric surfaces.}
The purpose of this paper is to generalize the previous results to the
context of counting curves on a large family of (possibly non-smooth)
toric surfaces $\Tor(\cc)$, 
which includes $\PP^1
\times \PP^1$ and Hirzebruch surfaces. A new and interesting feature of our results is that the Severi degree $N^{\dd, \delta}_{\Tor(\cc)}$ of such a toric surface $\Tor(\cc)$ is a polynomial not only as a function of the degree $\dd$, but also as a function of $\cc$, i.e., as a ``function of the surface" itself.
\medskip

\noindent \textbf{A note for combinatorialists.} A familiarity with the basic facts of toric geometry is desirable to understand the motivation for this work (and we refer the reader to Fulton's book \cite{Fu93} for the necessary definitions and background information). However, the machinery of tropical geometry allows for a purely combinatorial approach to studying Severi degrees, and most of this paper can be read independently of that background.

\medskip

We now state our results more precisely. 

\begin{notation}\label{not:polygon}
A polygon $P$ is said to be \emph{$h$-transverse} if it has integer coordinates and every edge has slope $0$, $\infty$, or $\frac1k$ for some integer $k$. Let $d^{\,t}$ and $d^{\,b}$ be the lengths of the top and bottom edges of $P$, if they exist (and 0 if they don't exist). Let the edges on the right side of the polygon, listed clockwise from top to bottom, have directions $(c^r_1, -1), \ldots, (c^r_n,-1)$ and lattice lengths $d^{\,r}_1, \ldots, d^{\,r}_n$,  so $c^r_1 > \cdots > c^r_n$.
Let the edges on the left side of the polygon, listed counterclockwise from top to bottom, have directions $(c^l_1, -1), \ldots, (c^l_m,-1)$ and lattice lengths $d^{\,l}_1, \ldots, d^{\,l}_m$,  so $c^l_1 < \cdots < c^l_m$. Notice that $d^{\,t}+\sum_i c^r_i d^{\,r}_i - d^{\,b} - \sum_j c^l_j d^{\,l}_j =0$.

Denote $\cc^r = (c^r_1, \ldots, c^r_n)$,  $\dd^r = (d^{\,r}_1, \ldots, d^{\,r}_n)$, $\cc^l = (c^l_1, \ldots, c^l_m)$,  $\dd^l = (d^{\,l}_1, \ldots, d^{\,l}_m)$, and $\cc=(\cc^r ; \cc^l), \dd= (d^{\,t} ; \dd^r ; \dd^{l})$. Finally, denote $\Delta(\cc, \dd) := P$. Observe that $\cc$ is the set of slopes of the non-vertical rays in the normal fan of $\Delta(\cc, \dd)$.
\end{notation}

%
Figure \ref{fig:bigpolygon} shows the polygon $\Delta(\cc, \dd)$ with $\cc=((3,1,0,-1);(-2,0,1,2))$ and \,\,  $\dd=(1;(1,2,1,1); (1,1,1,2))$ and its normal fan.

\begin{figure}[h]
\begin{center}
\qquad \includegraphics[height=4cm]{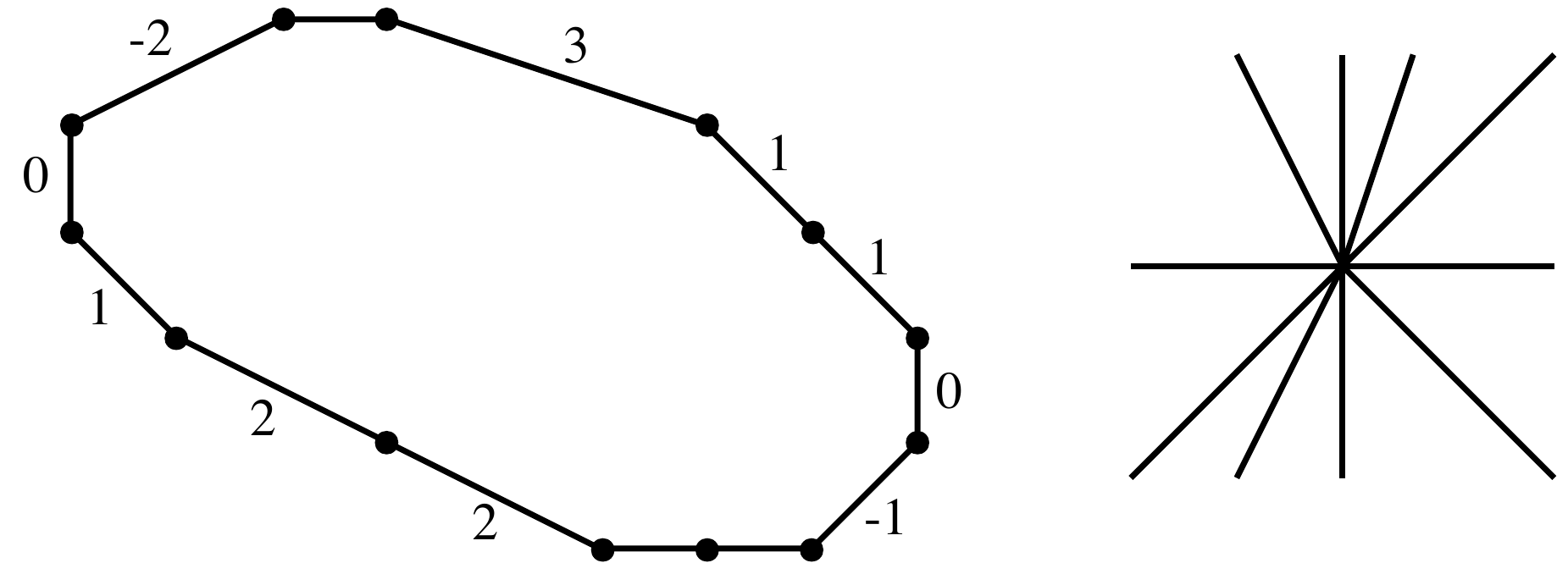}
\end{center}
\caption{An $h$-transverse polygon and its normal fan.}
\label{fig:bigpolygon}
\end{figure}

The normal fan of the polygon $\Delta(\cc, \dd)$ consists of the outward rays centered at the origin and perpendicular to the sides. This fan 
determines a projective toric surface~$\Tor(\cc)$ (which only depends on $\cc$ and whether $d^{\,t}$ and $d^{\,b}$ are zero).
Additionally, the polygon itself determines an ample line bundle
$\L_\cc(\dd)$ on $\Tor(\cc)$; let $|\L_\cc(\dd)|$ be the complete
linear system of divisors on $\Tor(\cc)$ 
corresponding to $\L_\cc(\dd)$. 


When we count curves on $\Tor(\cc)$, we will loosely think of $\cc$ as the surface where our curves live, and $\dd$ as their \emph{multidegree}. This is motivated by the case when $\Delta(\cc,\dd) = \textrm{conv}\{(0,0), (m,0), (0,m)\}$. In this case the toric surface is $\PP^2$, and the linear system $|\L_\cc(\dd)|$ parameterizes the degree $m$ curves on $\PP^2$. 

Given a positive integer $\delta$, the \emph{Severi variety} is
the closure of the set of $\delta$-nodal curves in 
$|\L_{\cc}(\dd)|$. 
Its degree is the \emph{Severi
  degree} $N_{\Tor(\cc)}^{\dd, \delta}$. 
This number also counts:
\begin{itemize}
\item
the $\delta$-nodal curves in $|\L_\cc(\dd)|$ which pass through a given set of $|\Delta \cap \ZZ^2| -1 - \delta$ generic
points in 
$\Tor(\cc)$, and 
\item
the $\delta$-nodal curves in the torus $(\CC^*)^2$ defined by polynomials with Newton
polygon $\Delta(\cc,\dd)$ which go through a given set of $|\Delta \cap \ZZ^2| -1-\delta$ generic
points in $(\CC^*)^2$.
\end{itemize}
%

Our main result is that, for a fixed number of nodes $\delta$, \textbf{the
Severi degree $N_{\Tor(\cc)}^{\dd, \delta}$ is a polynomial in both $\cc$ and
$\dd$}, provided $\cc$ and $\dd$ are sufficiently large and ``spread out", in the precise sense defined below.

\begin{theorem} (Polynomiality of Severi degrees 1: Fixed Toric Surface.)\label{mainthm:fixed}
\\
\noindent Fix $m,n \ge 1$, $\delta \ge 1$, and $\cc \in \ZZ^{m+n}$. There is a combinatorially defined polynomial
$p^\cc_\delta(\dd)$ such that the Severi degree $N_{\Tor(\cc)}^{\dd, \delta}$ is
  given by
\begin{equation}
\label{eqn:polynomiality}
N_{\Tor(\cc)}^{\dd, \delta} = p^\cc_\delta(\dd)
\end{equation}
for any sufficiently large and spread out $\dd \in
\ZZ_{\ge 0}^{m+n+1}$.

More precisely, the result holds if we assume, in Notation \ref{not:polygon}, that:
\begin{eqnarray*}
d^{\,t}, d^{\,b} & \geq & \delta, \\
d^{\,t}+c^r_1-c^l_1, \,\, d^{\,b}+c^r_n-c^l_m &\geq & 2\delta, \\
 d^{\,r}_i, \,\,  d^{\,l}_j &\geq& \delta+1 \quad (1 \leq i \leq n, \,\, 1 \leq j \leq m),
\\
\left|(d^{\,r}_1+\cdots+d^{\,r}_i) - (d^{\,l}_1+\cdots+d^{\,l}_j)\right| &\geq&  \delta+2 \quad (1 \leq i \leq n-1, \,\, 1 \leq j \leq m-1).
\end{eqnarray*}

\end{theorem}

\begin{theorem} (Polynomiality of Severi degrees 2: Universality.)\label{mainthm:universal}\\
\noindent Fix $m,n \ge 1$ and $\delta \ge 1$. There is a universal and combinatorially defined polynomial $p_\delta(\cc,\dd)$ such that the Severi degree $N_{\Tor(\cc)}^{\dd, \delta}$ is
  given by
\begin{equation}
\label{eqn:polynomiality}
N_{\Tor(\cc)}^{\dd, \delta} = p_\delta(\cc,\dd)
\end{equation}
for any sufficiently large and spread out $\cc \in \ZZ^{m+n}$ and $\dd  \in
\ZZ_{\ge 0}^{m+n+1}$.

More precisely, the result holds if we assume, in Notation \ref{not:polygon}, that:
\begin{eqnarray*}
d^{\,t}, d^{\,b} & \geq & \delta, \\
d^{\,t}+c^r_1-c^l_1, \,\, d^{\,b}+c^r_n-c^l_m &\geq & 2\delta, \\
 d^{\,r}_i, \,\,  d^{\,l}_j &\geq& \delta+1 \quad (1 \leq i \leq n, \,\, 1 \leq j \leq m), \\
\left|(d^{\,r}_1+\cdots+d^{\,r}_i) - (d^{\,l}_1+\cdots+d^{\,l}_j)\right| &\geq&  \delta+2 \quad (1 \leq i \leq n-1, \,\, 1 \leq j \leq m-1), \\
c^r_i - c^r_{i+1}, \,\, c^l_{j+1} - c^l_j   &\geq& \delta+1  \quad (1 \leq i \leq n-1, \,\, 1 \leq j \leq m-1). 
\end{eqnarray*}
\label{thm:maintheoremuniversal}
\end{theorem}

Naturally, we have $p^\cc_\delta(\dd) = p_\delta(\cc,\dd)$ as polynomials in $\dd$ for all sufficiently spread out $\cc$ (in the sense of the last condition of Theorem \ref{mainthm:universal}).

As special cases, we obtain polynomiality results for curve counts on
$\PP^1 \times \PP^1$, and Hirzebruch surfaces. 
In Remark \ref{rmk:implicitHirzebruch} we compute the node polynomials $p_\delta$ for $\delta \leq 5$ for
Hirzebruch
surfaces.
Similar results hold
for  toric surfaces arising from polygons with
one or no horizontal edges, such as $\PP^2$; see Remark~\ref{rmk:fewerhorizontaledges}.

The restriction in Theorems~\ref{mainthm:fixed}
and~\ref{mainthm:universal} to toric surfaces of $h$-transverse
polygons is a technical assumption necessary for our proof. Floor
diagrams, our main combinatorial tools, are (as of now) only defined
in this situation. We expect, however, that similar results
hold for arbitrary toric surfaces. 

%





\subsection{The relationship with G\"ottsche's Conjecture.}
\label{sec:relationwithGoettscheConjecture}

Our work is closely related to G\"ottsche's Conjecture \cite[Conjecture.~2.1]{Go}, 
but the precise relationship still requires clarification. 
G\"ottsche conjectured the existence of
universal polynomials $T_\delta(w, x, y, z)$ that compute the Severi degree for any smooth projective
algebraic surface $S$ and any sufficiently ample line bundle $\L$ on $S$. According to the
conjecture, the number of
$\delta$-nodal curves in the linear system $|\L|$ through an
appropriate number of points is given by evaluating $T_\delta$ at
the four topological numbers 
$\L^2, \L K_S, K_S^2$ and $C_2(S)$. 
Here $K_S$ denotes the
canonical bundle, $C_1$ and $C_2$ represent Chern classes, and $LM$ denotes the degree of $C_1(L) \cdot C_1(M)$ for line bundles $L$ and $M$. 
 Recently,
Tzeng~\cite{Tz10} proved G\"ottsche's Conjecture.

If the toric surface $\Tor(\cc)$ is smooth, then all four topological numbers
mentioned above are polynomials in $\cc$ and $\dd$.
In that case, Tzeng's proof of G\"ottsche's conjecture implies that, for fixed $\delta$, the Severi
degrees $N_{\Tor(\cc)}^{\dd, \delta}$ are given by a universal polynomial in $\cc$ and $\dd$,
provided that $\L_\cc(\dd)$ is $(5 \delta -1)$-ample.

G\"ottsche's conjecture does not imply our results because the toric surfaces considered in Theorems~\ref{mainthm:fixed} and~\ref{mainthm:universal} are almost never smooth. 
The surface $\Tor(\cc)$ is smooth precisely when any two adjacent rays in the normal fan span the lattice $\ZZ^2$. This happens if and only if
$c_1^r-c_2^r = \cdots = c_{n-1}^r - c_n^r = 1$ and $c_1^l-c_2^l =
\cdots = c_{m-1}^l - c_m^l = -1$.

It is natural to ask for a generalization of the four topological
numbers to some class of singular surfaces, so that G\"ottsche's universal polynomial
 $T_\delta(w, x, y, z)$ specializes to the polynomial of
 Theorem~\ref{mainthm:universal}. We do not know how to do
 this, even for  $\Tor(\cc)$ with $\cc = ((c_1^r, c_2^r); (0,0))$. In general,
  $C_1(K_S)$ is defined for any projective
  variety~\cite[Section~3]{Fu84}.
To define
$C_2(S)$ for singular $S$, one could pass to
MacPherson's Chern class~\cite{Ma74}, which is defined for any
algebraic variety. Alternatively, for
toric surfaces, $C_2(S)$ could also be defined via the combinatorial
formula for the Chern polynomial of a toric variety.
However, we checked that
  $T_\delta(w, x, y, z)$, evaluated at
  any of the proposed sets of numbers, gives a different
  polynomial and that further (topological) correction terms appear to
  be necessary
  (c.f.\ Section~\ref{sec:non-smooth-example}). Alternatively, evaluating $T_\delta(w, x, y, z)$ at the topological
  numbers of a smooth resolution of $\Tor(\cc)$ does not yield the
  desired result either.

Still, the Severi degrees are uniformly given by a polynomial in $\cc$
and $\dd$, provided $\dd$ is sufficiently large.
This suggests that it may be possible to generalize G\"ottsche's
conjecture to a class of singular algebraic surfaces. We present some
numerical data in Section~\ref{sec:non-smooth-example}.


\subsection{Outline}
This paper is organized as follows. In Section \ref{sec:floordiagrams}
we describe Brugall\'{e} and Mikhalkin's tropical method for counting
irreducible curves on ``$h$-transverse" toric surfaces in terms of
floor diagrams, and generalize it to compute Severi degrees of such surfaces. We then focus on the combinatorial enumeration of floor diagrams. In Section \ref{sec:templates} we generalize Fomin and Mikhalkin's ``template decomposition" of a floor diagram, and express Severi degrees in terms of templates. The resulting formula is intricate and not obviously polynomial. In Sections \ref{sec:polynomialityspecial} and \ref{sec:polynomialitygeneral} we analyze this formula in detail. After several simplifying steps, we express Severi degrees as a finite sum, where each summand is a ``discrete integral" of a polynomial function over a variable polytopal domain. This allows us to prove their eventual polynomiality. In Section \ref{sec:polynomialityspecial} we carry this out for ``first-quadrant" $h$-transverse toric surfaces, and in Section \ref{sec:polynomialitygeneral} we do it for  general $h$-transverse toric surfaces. The formulas we obtain for Severi degrees are somewhat complicated, but they are completely combinatorial and effectively computable. We illustrate this in Section~\ref{sec:computations}
by computing the Severi degree of
any large enough Hirzebruch surface for $\delta \leq 5$ nodes as well as
for a family of singular toric surfaces and $\delta \leq 2$.

\bigskip

{\bf Acknowledgements.} We thank the referee for valuable comments
that helped improve this article. We also
thank Beno\^it Bertrand and Erwan Brugall\'e for clarifying some details
about floor
diagrams.  The second author thanks Eugene Eisenstein,
Jos\'e Gonz\'alez, Ragni Piene, and Alan Stapledon for explaining some
facts about toric surfaces and Chern classes, and Bill Fulton for valuable comments on
the subject. Part of this work was accomplished at the MSRI (Mathematical
  Sciences Research Institute) in Berkeley, CA, USA, during the
Fall 2009 semester in tropical geometry. We would like to
  thank MSRI for their hospitality.

\section{Counting curves with floor diagrams}
\label{sec:floordiagrams}

In this section we review the floor diagrams of Brugall\'e and
Mikhalkin \cite{BM1, BM2} associated to curves on toric surfaces which come from
$h$-transverse polygons. We will introduce them using notation inspired by Fomin and Mikhalkin \cite{FM} who discussed floor diagrams in the
planar case. 

\subsection{An outline of the tropical method for counting curves.}\label{sec:tropicalmethod}
We briefly sketch Brugall\'e and Mikhalkin's technique for counting
curves through a generic set of points on a toric
surface. \cite{BM1,BM2} (They carried this out for irreducible curves,
and we extend it to possibly reducible curves and Severi degrees.)  We wish to count the
$\delta$-nodal curves in $(\CC^*)^2$, having Newton polygon $\Delta$,
which go through sufficiently many generic points. Brugall\'e and
Mikhalkin proved that this problem can be ``tropicalized": we can
``just" count the $\delta$-nodal \emph{tropical curves} with Newton
polygon $\Delta$ going through a generic set of points. Roughly
speaking, such a tropical curve is an edge weighted polyhedral complex in the plane which is dual to a polyhedral subdivision of $\Delta$, as shown in Figure \ref{fig:tropicalcurve}. 

\begin{figure}[h]
\begin{center}
\qquad \includegraphics[height=7.1cm]{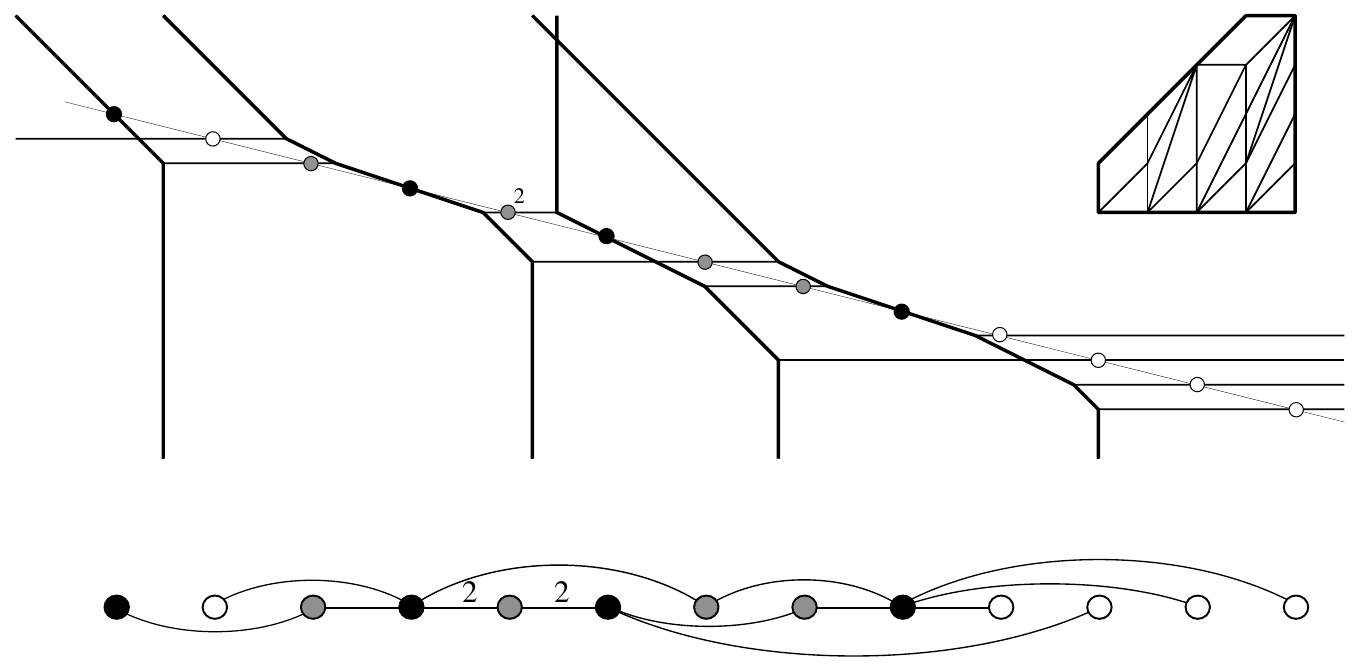}
\end{center}
\caption{A tropical curve through a vertically stretched collection of 13 points (rotated $90^\circ$ counterclockwise), the dual subdivision of the Newton polygon (also rotated $90^\circ$ counterclockwise), and the corresponding marked floor diagram.}
\label{fig:tropicalcurve}
\end{figure}

The resulting tropical enumeration problem is still very subtle. When
the polygon $\Delta$ is $h$-transverse, it can be simplified. One can
assume that the generic points $P$ lie very far from each other on an
almost vertical line, i.e., are ``vertically stretched''. 
In this case, one can control where the points of $P$ must land on the tropical line $L$. Divide the curve into \emph{elevators}, which are all the vertical segments of $L$ (they are horizontal in Figure \ref{fig:tropicalcurve}) and \emph{floors}, which are the connected components of $L$ upon removal of the elevators (they are bold in Figure \ref{fig:tropicalcurve}). The $h$-transversality condition then guarantees that one must have exactly one point of $P$ on each elevator and exactly one on each floor.

That geometric incidence information is then recorded in a \emph{floor diagram}. This diagram has a node for each floor of $L$, and an edge for each elevator connecting two floors. More detailed information is contained in the \emph{marked floor diagram}. This diagram has one black node for each floor of $L$, and one gray/white node for each bounded/unbounded elevator. Its edges show how the elevators connect the different floors of $L$. 

This correspondence encodes all the necessary geometric information into combinatorial data, and reduces the computation to a (still very subtle) purely enumerative problem on marked floor diagrams. We now define marked floor diagrams precisely, and explain how exactly we need to count them.

\subsection{Counting irreducible curves via connected floor diagrams.}\label{sec:gromovwitten}
  Given a lattice polygon $\Delta$ and a positive integer $\delta$, let 
 $N_{\Delta,\delta}$ 
 be the number of
$\delta$-nodal \emph{irreducible} curves in the 
torus $(\CC^*)^2$, given by polynomials with Newton
polygon $\Delta$, which go through $|\Delta \cap \ZZ^2| -1-\delta$ generic
points in $(\CC^*)^2$. 
When the toric surface is Fano, these are the
\emph{Gromov-Witten invariants} of the surface $\Tor(\Delta)$. 
(These numbers should not be confused with the closely related $N^{\Delta,\delta}$, which counts curves that are not necessarily irreducible.)

%
We now explain how these numbers can be computed tropically, following Brugalle and Mikhalkin's work. \cite{BM1, BM2}

For the rest of the paper we will assume that $\Delta=\Delta(\cc,\dd)$ is $h$-transverse and we will use Notation \ref{not:polygon} to describe it. Define the multiset $D_r$ of \emph{right directions} of $\Delta$ to be the multiset containing each right direction $c^r_i$ repeated $d^{\,r}_i$ times. Define $D_l$ analogously. The cardinality of $D_r$ (or, equivalently, of $D_l$) is the \emph{height} of $\Delta$.

%
%

\begin{example}
For the polygon $\Delta$ of Figure \ref{fig:bigpolygon}, the multisets of left directions and right directions are $D_l=\{-2,0,1,2,2\}$ and $D_r=\{3,1,1,0,-1\}$ and the upper edge length is $d^{\,t}=1$. 
\end{example}

Fix an $h$-transverse polygon $\Delta$. Now we define the combinatorial
objects which, when weighted correctly, compute $N_{\Delta, \delta}$.

\begin{definition}
A \emph{$\Delta$-floor diagram} $\D$ 
consists of:
\begin{itemize}
\item
two permutations\footnote{The permutations of a multiset are counted without repetition. For instance, the multiset $\{1,1,2\}$ has three permutations: $(1,1,2), (1,2,1), (2,1,1)$.}
 $(l_1, \dots l_M)$ and $(r_1, \dots r_M)$ of the multisets $D_l$ and $D_r$ of left and right directions of $\Delta$, and a sequence $(s_1, \dots s_M)$ of non-negative integers such that  $s_1+ \cdots+ s_M = d^{\,t}$, 
\item
a graph on a vertex set $\{1, \dots, M \}$, possibly with multiple edges, with edges directed $i \to j$ for $i<j$, and
\item
edge weights $w(e) \in \ZZ_{>0}$ for all edges $e$ such that
for every vertex $j$,
\begin{displaymath}
\dive(j) := \sum_{ \tiny
     \begin{array}{c}
  \text{edges }e\\
j \stackrel{e}{\to} k
     \end{array}
} w(e) -   \sum_{ \tiny
     \begin{array}{c}
  \text{edges }e\\
i \stackrel{e}{\to} j
     \end{array}
} w(e) \leq r_j - l_j + s_j. \,  \footnote{This inequality will become clear when we define the markings of a floor diagram.}
\end{displaymath}

\end{itemize}

Sometimes we will omit $\Delta$ and call $\D$ a \emph{toric floor diagram} or simply a \emph{floor diagram}. When a floor diagram has $\ll=(l_1, \ldots, l_M), \rr=(r_1, \ldots, r_M), \ss=(s_1, \ldots, s_M)$, we will call it an $(\ll,\rr,\ss)$-floor diagram. We will also call $\aa:=(d^{\,t},\rr-\ll)$ the \emph{divergence sequence}, because in Definition \ref{def:marking} we will add some edges to obtain a diagram $\tilde{\D}$ with this vertex divergence sequence, and it is this new diagram that we will mostly be working with.
\end{definition}

\begin{example}
Figure \ref{fig:floordiagram} shows the toric floor diagram corresponding to the tropical curve of Figure \ref{fig:tropicalcurve}, with $D_l=\{0,0,0,0\}, D_r=\{1,1,1,0\},$ and $I_0 = 1$. We place the vertices on a line in increasing order and omit the (left-to-right) edge directions.
\end{example}

\begin{figure}[h]
\begin{flushleft}
\qquad\qquad\qquad \includegraphics[height=0.8cm]{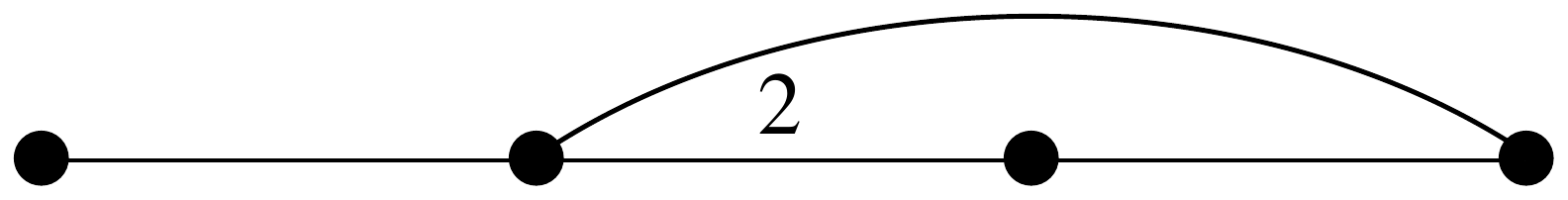}
\caption{A toric floor diagram.}
\label{fig:floordiagram}
\end{flushleft}
\vspace{-2.5cm}
\begin{small}
\begin{eqnarray*}
\qquad\qquad\qquad\qquad\qquad \qquad\qquad\qquad\qquad \qquad\qquad\mathbf{r} &=& (1,1,0,1)\\
\mathbf{l} &=& (0,0,0,0)\\
\mathbf{s} &=& (0,1,0,0)
\end{eqnarray*}
\end{small}
\end{figure}

A floor diagram $\D$ is \emph{connected} if its underlying graph
is. Notice that in \cite{BM2} floor diagrams are necessarily
connected; we don't require that. The \emph{genus} of $\D$ is the genus $g(\D)$ of the underlying graph (or
the first Betti number of the underlying topological space). If $\D$
is connected its \emph{cogenus} is given by
\begin{displaymath}
\delta(\D) = |\inte(\Delta) \cap \ZZ^2| - g(\D),
\end{displaymath}
where $\inte(\Delta)$ denotes the interior of the polygon $\Delta$.
This definition is motivated by the fact that an irreducible algebraic curve
of genus $g$ with $\delta$ nodes and Newton polygon $\Delta$ satisfies
$\delta + g = |\inte(\Delta) \cap \ZZ^2 |$. 
Via the correspondence between algebraic curves and floor diagrams
(see \cite{BM2}) these notions literally correspond to the respective
analogues for algebraic curves. Connectedness corresponds to
irreducibility. 

Lastly, a floor diagram $D$ has \emph{multiplicity}
\begin{displaymath}
\mu(\D) = \prod_{\text{edges }e} w(e)^2.
\end{displaymath}



To enumerate algebraic curves via floor diagrams we need to count certain
markings of these diagrams, which we now define.

\begin{definition}\label{def:marking}
A \emph{marking} of a floor diagram $\D$ is defined by the following
four step process. 

{\bf Step 1:} For each vertex $j$ of $\D$, create $s_j$ new indistinguishable vertices and
connect them to $j$ with new edges directed towards $j$. 

{\bf Step 2:} For each vertex $j$ of $\D$, create $r_j - l_j + s_j- \textrm{div}(j)$ new indistinguishable vertices and
connect them to $j$ with new edges directed away from $j$. This makes the divergence of vertex $j$ equal to $r_j - l_j$ for $1 \leq j \leq M$.

{\bf Step 3:} Subdivide each edge of the original floor
diagram $\D$ into two
directed edges by introducing a new
vertex for each edge. The new edges inherit their weights and orientations. Denote the
resulting graph $\tilde{\D}$.

\begin{figure}[h]
\begin{flushleft}
\qquad \includegraphics[height=1.5cm]{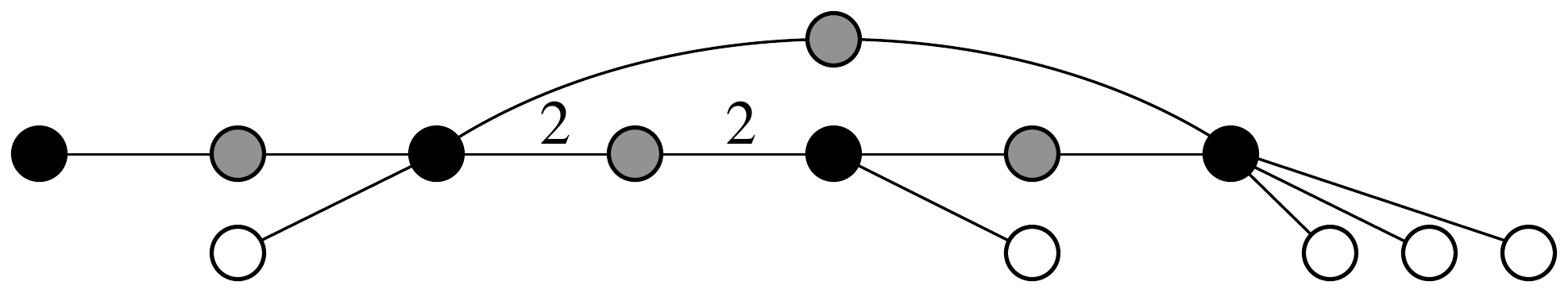}
\caption{The result of applying Steps 1-3 to Figure \ref{fig:floordiagram}.}
\label{fig:steps1-3}
\end{flushleft}
\vspace{-3cm}
\begin{small}
\begin{eqnarray*}
\qquad \qquad\qquad\qquad\qquad \qquad\qquad\qquad\qquad\qquad\qquad\qquad \qquad\qquad\mathbf{r} &=& (1,1,0,1)\\
\mathbf{l} &=& (0,0,0,0)\\
\mathbf{s} &=& (0,1,0,0)\\
\mathbf{r - l} &=& (1,1,0,1)
\end{eqnarray*}
\end{small}
\end{figure}

{\bf Step 4:} Linearly order the vertices of $\tilde{\D}$
extending the order of the vertices of the original floor
diagram $\D$ such that, as before, each edge is directed from a
smaller vertex to a larger vertex.

The extended graph $\tilde{\D}$ together with the linear order on
its vertices is called a \emph{marked floor diagram}, or a
\emph{marking} of the original floor diagram $\D$.
\begin{figure}[h]
\begin{center}
\qquad \includegraphics[height=1.2cm]{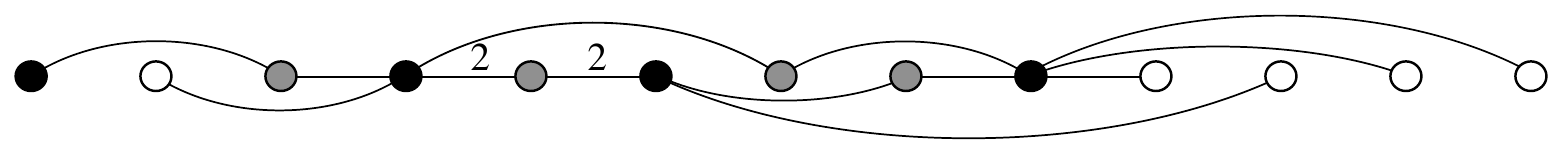}
\end{center}
\bigskip
$\mathbf{r} = (1,1,0,1), 
\mathbf{l} = (0,0,0,0), 
\mathbf{s} = (0,1,0,0), \quad 
\mathbf{r - l} = (1,1,0,1)$
\caption{A marking of the floor diagram of Figure \ref{fig:floordiagram}.}
\label{fig:markedfloordiagram}
\end{figure}

\end{definition}

Tropically, Step 1 corresponds to adding the upward elevators to get the right Newton polygon, Step 2 corresponds to adding the downward elevators to balance each floor, Step 3 marks the bounded elevators, and Step 4 decides the order in which the given points land on the floors and elevators of our curve.

Keeping in mind that we introduced indistinguishable vertices in Steps 1 and 2, we need to count marked floor diagrams up to equivalence. Two such $\tilde{\D}_1$, $\tilde{\D}_2$ are \emph{equivalent} if
$\tilde{\D}_1$ can be obtained from $\tilde{\D}_2$ by permuting edges
without changing their weights; \emph{i.e.,} if there exists an
automorphism of weighted graphs which preserves the vertices of $\D$ and maps $\tilde{\D}_1$ to
$\tilde{\D}_2$.  The \emph{number of markings} $\nu(\D)$ is the number of
marked floor diagrams $\tilde{\D}$ up to equivalence.

\begin{example}
Let us compute $\nu(\D)$ for the floor diagram of Figure \ref{fig:floordiagram}, by counting the possible linear orderings of Figure \ref{fig:steps1-3}. Modulo isomorphism, the ordering of ten of the vertices is fixed. The leftmost lower white vertex can be inserted in three places. The top gray vertex can be placed in 4 positions. For two of them, the second white vertex can be placed in 6 positions, while for the other two it can be placed in 7 positions. Therefore $\nu(\D) = 3(2 \cdot 6 + 2 \cdot 7) = 78$.
\end{example}

We can now combinatorialize the problem of counting irreducible curves
with given genus and $h$-transverse Newton polygon.

\begin{theorem}\label{th:BM}[Theorem 3.6 of \cite{BM2}]
For any $h$-transverse polygon $\Delta$ and any $\delta \ge 0$, 
the number of irreducible curves  in the 
torus $(\CC^*)^2$, having $\delta$ nodes and Newton polygon~$\Delta$, and going through $|\Delta \cap \ZZ^2| -1-\delta$ given generic
points in $(\CC^*)^2$, equals
\begin{displaymath}
N_{\Delta, \delta} = \sum_\D \mu(\D) \nu(\D),
\end{displaymath}
where the sum runs over all connected $\Delta$-floor diagrams 
of cogenus $\delta$.

\end{theorem}

Brugall\'e's and Mikhalkin's definition of floor diagram slightly differs
from ours, but it records the same information. (Our $s_j$ is their
number of edges in ``$Edge^{+\infty}$'' adjacent to vertex $j$. Our $l_j$ is their
``$- \theta(j)$'' and our $r_j$ is their ``$\dive(j)$'').

\subsection{Severi degrees: Counting (possibly reducible) curves via (possibly disconnected) floor diagrams.}

We now extend Brugall\'{e} and Mikhalkin's result of the previous section to curve counts of possibly reducible curves. We are now interested in the \emph{Severi degree} $N^{\Delta, \delta}$, which is the number of (possibly reducible) $\delta$-nodal curves in the torus $(\CC^*)^2$, given by polynomials with Newton polygon $\Delta$, which pass through $|\Delta \cap \ZZ^2| -1-\delta$ generic points in $(\CC^*)^2$.

Severi degrees equal the numbers computed in the previous section when $\delta$ is small, and can be expressed in terms of them for any $\delta$, as we now explain, paralleling \cite[Section 1]{FM}. 

We wish to count $\delta$-nodal curves with Newton polygon $\Delta$  through a 
given generic set $\Pi \subset (\CC^*)^2$ of $|\Delta \cap \ZZ^2| - 1- \delta$ points.
Let $\Pi_1, \dots, \Pi_t$ be a partition of $\Pi$ into some number of subsets; and for each $1 \leq i \leq t$, let $C_i$ be an irreducible $\delta_i$-nodal curve with Newton
polygon $\Delta_i$ passing through the points in
$\Pi_i$, where
\begin{equation}
\label{eq:cardinality}
|\Pi_i| = |\Delta_i \cap \ZZ^2| - 1- \delta_i,
\end{equation}
The curve $C = C_1 \cup \dots
\cup C_t$ has the correct Newton polygon $\Delta$ if
\begin{equation}
\label{eq:minkowskisum}
\Delta = \Delta_1 + \cdots + \Delta_t.
\end{equation}
Also, Bernstein's theorem \cite{Be75} tells us that the number of intersection points of $C_i$ and $C_j$ is  the \emph{mixed area} $\M(\Delta_i,\Delta_j):= \frac12(Area(\Delta_i+\Delta_j)-Area(\Delta_i)-Area(\Delta_j))$ of their Newton polygons. Therefore, $C$ has the right number of nodes if 
\begin{equation}
\label{eq:nodeequation}
\delta = \sum_{i=1}^t \delta_i + \sum_{1 \le i < j \le t} \M(\Delta_i, \Delta_j).
\end{equation}
The sum $\sum_{i < j} \M(\Delta_i, \Delta_j)$ is denoted $\M(\Delta_1, \dots, \Delta_t)$ and called the mixed area of
the polygons $\Delta_1, \dots, \Delta_t$. It is easily computed in terms of the sides of the $\Delta_i$s.
%

The previous argument tells us how to express Severi degrees in terms of the numbers of the previous section.
We have
\begin{equation}
\label{eq:GromovtoSeveri}
N^{\Delta, \delta} = \sum_{\Pi = \cup \Pi_i} \sum_{(\Delta_i, \delta_i)}
\prod_i N_{\Delta_i, \delta_i},
\end{equation}
where the first sum is over all partitions of $\Pi$, and the second sum
is over all pairs $(\Delta_i, \delta_i)$ which satisfy
(\ref{eq:cardinality}), 
(\ref{eq:minkowskisum}), and (\ref{eq:nodeequation}). 
In particular, when the polygon $\Delta$ is large enough that $\delta < \M(\Delta_1, \dots, \Delta_t)$ for any nontrivial Minkowski sum decomposition $\Delta = \Delta_1 + \cdots + \Delta_t$, we have $t=1$ and $N^{\Delta,\delta} = N_{\Delta, \delta}$. 


A similar analysis holds at the level of floor diagrams. Let $\D$ be a (non necessarily connected) floor diagram.
Let $V(\D)= \cup_{i=1}^t V_i$ be the partition of
the vertices of $\D$ given by the connected components of $\D$, and let $\D_1, \ldots, \D_t$ be the corresponding (connected) floor diagrams. Define
$h$-transverse polygons $\Delta_i$ (for $1 \leq i \leq t)$ by the collections $\{(l_j, r_j, s_j) \}_{j \in V_i}$, where in each such collection
the index $j$
runs over the vertices $j$ in $V_i$. 
Finally define $\delta(\D) = \sum_i \delta(\D_i) + \M(\Delta_1, \ldots, \Delta_t)$. It is not hard to write an explicit expression for $\delta(\D)$. 
Theorem \ref{th:BM} and (\ref{eq:GromovtoSeveri}) give:

\begin{theorem}
\label{cor:keycorollary}
For any $h$-transverse polygon $\Delta$ and any $\delta \ge 0$ the Severi
degree $N^{\Delta, \delta}$ is given by
\begin{equation}
N^{\Delta, \delta} = \sum \mu(\D) \nu(\D),
\tag{{\color{blue}Severi1}}\label{eq:Severi1}
\end{equation}
summing over all (not necessarily connected) $\Delta$-floor diagrams 
$\D$ of cogenus $\delta$.
\end{theorem}

%
%
%
%
%
%

\begin{example}
For $\Delta = \conv\{(0,0), (0,2), (2,2), (4,0)\}$ and $\delta = 1$, one can check that there are three floor diagrams, and Theorem \ref{cor:keycorollary} gives 
\[
N^{\Delta,1} = 1 \cdot 7 + 1 \cdot 5 + 4 \cdot 2 = 20.
\]
For $\Delta' = \conv\{(0,0), (2,0), (2,2), (0,4)\}$ and $\delta = 1$, we get
\[
N^{\Delta', 1} = 1 \cdot 4 + 1 \cdot 4 + 1 \cdot 3 + 4 \cdot 1 + 4 \cdot 1 + 1 \cdot 1 = 20
\]
Notice that, by choosing to count tropical curves through a vertically stretched configuration, we have broken the symmetry between $\Delta$ and $\Delta'$. 
\end{example}

Equation (\ref{eq:Severi1}) is the first in a series of combinatorial formulas for the Severi degree $N^{\Delta,\delta}$, which we will use to prove the eventual polynomiality of $N^{\Delta, \delta}$. While the right hand side is certainly combinatorial, it is unmanageable in several ways.  The first difficulty is that the indexing set is terribly complicated. The following section provides a first step towards gaining control over it.

%

\section{Template decomposition of floor diagrams and  Severi degrees}\label{sec:templates}

We now introduce a decomposition of the
floor diagrams of Section \ref{sec:floordiagrams} into ``basic
building blocks'', called \emph{templates}. This extends earlier work of 
Fomin and Mikhalkin~\cite{FM}
who did this in the planar case.

\subsection{Templates.} 



\begin{definition}  \cite[Definition 5.6]{FM}.
\label{def:template}
A \emph{template} $\Gamma$ is a directed graph on vertices
$\{0, \dots, l\}$, where 
$l\ge 1$, with
possibly multiple edges and edge
weights $w(e) \in \ZZ_{>0}$, satisfying:
\begin{enumerate}
\item If $i \stackrel{e} {\to} j$ is an edge then $i <j$.
\item Every edge $i \stackrel{e}{\to} i+1$ has weight $w(e) \ge
  2$. (No ``short edges''.)
\item For each vertex $j$, $1 \le j \le l-1$, there is an edge
  ``covering'' it, i.e., there exists an edge $i \stackrel{e}{\to} k$ with $i <j <k$.
\end{enumerate}
\end{definition}

\begin{figure}[htbp]
\begin{center}
\begin{tabular}{c|c|c|c|c|c|c}
$\Gamma$ &
$\delta(\Gamma)$ & $l(\Gamma)$ 
& $\mu(\Gamma)$ 
& $\eps_0(\Gamma)$ & $\eps_1(\Gamma)$ & $\varkappa(\Gamma)$ 
\\
\hline \hline
&&&&&&\\[-.1in]
\begin{picture}(95,8)(-10,-4)\setlength{\unitlength}{2.5pt}\thicklines
\oo
\put(5,2){\makebox(0,0){$\scriptstyle 2$}}
\Eeee
\end{picture}
& 1 & 1 & 4 & 0 & 0 & (2)
\\[.15in]
\begin{picture}(95,8)(-10,-4)\setlength{\unitlength}{2.5pt}\thicklines
\ooo
\qbezier(0.8,0.6)(10,5)(19.2,0.6)
\end{picture}
& 1 & 2 & 1 & 1 & 1 & (1,1) 
\\
\hline \hline
&&&&&&\\[-.1in]
\begin{picture}(95,8)(-10,-4)\setlength{\unitlength}{2.5pt}\thicklines
\oo
\put(5,2){\makebox(0,0){$\scriptstyle 3$}}
\Eeee
\end{picture}
& 2 & 1 & 9 & 0 & 0 & (3)
\\[.15in]
\begin{picture}(95,8)(-10,-4)\setlength{\unitlength}{2.5pt}\thicklines
\oo
\put(5,3.5){\makebox(0,0){$\scriptstyle 2$}}
\put(5,-3.5){\makebox(0,0){$\scriptstyle 2$}}
\qbezier(0.8,0.6)(5,2)(9.2,0.6)
\qbezier(0.8,-0.6)(5,-2)(9.2,-0.6)
\end{picture}
& 2 & 1 & 16 & 0 & 0 & (4)
\\[.15in]
\begin{picture}(95,8)(-10,-4)\setlength{\unitlength}{2.5pt}\thicklines
\ooo
\qbezier(0.8,0.6)(10,4)(19.2,0.6)
\qbezier(0.8,-0.6)(10,-4)(19.2,-0.6)
\end{picture}
& 2 & 2 & 1 & 1 & 1 & (2,2)
\\[.15in]
\begin{picture}(95,8)(-10,-4)\setlength{\unitlength}{2.5pt}\thicklines
\ooo
\qbezier(0.8,0.6)(10,4)(19.2,0.6)
\put(5,-2){\makebox(0,0){$\scriptstyle 2$}}
\Eeee
\end{picture}
& 2 & 2 & 4 & 0 & 1 & (3,1) 
\\[.15in]
\begin{picture}(95,8)(-10,-4)\setlength{\unitlength}{2.5pt}\thicklines
\ooo
\qbezier(0.8,0.6)(10,4)(19.2,0.6)
\put(15,-2){\makebox(0,0){$\scriptstyle 2$}}
\eEee
\end{picture}
& 2 & 2 & 4 & 1 & 0 & (1,3)
\\[.15in]
\begin{picture}(95,8)(-10,-4)\setlength{\unitlength}{2.5pt}\thicklines
\oooo
\qbezier(0.8,0.6)(15,6)(29.2,0.6)
\end{picture}
& 2 & 3 & 1 & 1 & 1 & (1,1,1)
\\[.15in]
\begin{picture}(95,8)(-10,-4)\setlength{\unitlength}{2.5pt}\thicklines
\oooo
\qbezier(0.8,0.6)(10,5)(19.2,0.6)
\qbezier(10.8,0.6)(20,5)(29.2,0.6)
\end{picture}
& 2 & 3 & 1 & 1 & 1 & (1,2,1)
\\[-.05in]
\end{tabular}
\end{center}
\caption{The templates with $\delta(\Gamma) \le 2$.}
\label{fig:templates}
\end{figure}

Every template $\Gamma$ comes with some numerical
data associated to it, which will play an important role later. 
Its \emph{length} $l(\Gamma)$ is the number of
vertices minus $1$. The product of squares of the edge weights
is its \emph{multiplicity} $\mu(\Gamma)$. Its \emph{cogenus} $\delta(\Gamma)$
is
\[
\delta(\Gamma) = \sum_{\stackrel{e}{i \to j}} \bigg[(j-i) w(e) -1
\bigg].
\]
For $1 \le j \le l(\Gamma)$ let $\varkappa_j = \varkappa_j(\Gamma)$
denote the sum of the weights of edges $i \stackrel{e}{\to} k$ with $i
< j \le k$, which we can think of as the flow over the midpoint between $j-1$ and $j$.
If $a_j(\Gamma)$ denotes the divergence of $\Gamma$ at vertex $j$, then $a_j(\Gamma) = \varkappa_{j+1} - \varkappa_j$, 
so we can also think of $\varkappa_j$ as the cumulative divergence to the left of $j$.

Lastly, set
 \begin{displaymath}
   \eps_0(\Gamma)  = \left\{
     \begin{array}{ll}
       1 &  \text{if all edges starting at } 0 \text{ have weight }1,\\
       0 &  \text{otherwise,}
     \end{array}
   \right.
\end{displaymath}
and
 \begin{displaymath}
   \eps_1(\Gamma)  = \left\{
     \begin{array}{ll}
       1 &  \text{if all edges arriving at } l \text{ have weight }1,\\
       0 &  \text{otherwise.}
     \end{array}
   \right.
\end{displaymath}
Figure~\ref{fig:templates} (courtesy of Fomin-Mikhalkin \cite{FM})
lists all templates $\Gamma$ with $\delta(\Gamma) \le 2$. Note that,
for any $\delta$, there are only finitely many templates with cogenus $\delta$.

\subsection{Decomposing a floor diagram into templates.}

We now show how to decompose a floor diagram $\D$ on vertices $1, \dots, M$ into templates. 
Recall that for each vertex $j$ of $\D$ we record a tuple of
integers $(l_j, r_j, s_j)$.

\begin{figure}[h]
\begin{center}
\qquad \includegraphics[height=1.2cm]{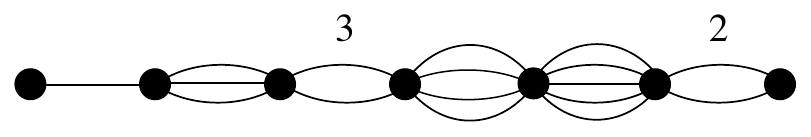}
\end{center}
\bigskip
$\mathbf{r} = (1,1,1,0,1,0,0)$, 
$\mathbf{l} = (0, 0,0,0,0,0,0)$,
$\mathbf{s} = (0,1,0,0,0,0,0)$

\caption{A floor diagram.}
\label{fig:bigfloordiagram}
\end{figure}

First, we add a vertex $0$ ($< 1$) to $\D$,
along with $s_j$ new edges of weight $1$
from $0$ to $j$ for each $1 \le j \le M$.
Then we add a vertex $M+1$ ($> M$),
together with $r_j - l_j + s_j - \dive(j)$ new edges
of weight $1$ from $i$ to $M+1$
for each $1 \le j \le M$.
The vertex divergence sequence of the resulting diagram $\D'$ is
$(d^{\,t}, r_1-l_1, \ldots, r_M-l_M, -d^{\,b})$. We drop the 
(superfluous) last entry from this sequence and as before we say
$(d^{\,t}, \rr-\ll)$ is the divergence sequence.

\begin{figure}[h]
\begin{center}
\qquad \includegraphics[height=1.5cm]{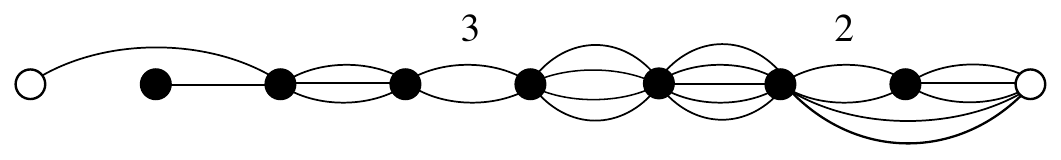}
\end{center}
\bigskip
$\mathbf{r-l} = (1,1,1,0,1,0,0)$, 
$\mathbf{s} = (0,1,0,0,0,0,0)$
\caption{The  floor diagram of Figure \ref{fig:bigfloordiagram}
with additional initial and final vertices.}
\label{fig:bigsatfloordiagram}
\end{figure}

Now remove all \emph{short edges}
from $\D'$, that is, all edges of weight 1 between consecutive
vertices. The result is an ordered collection of templates $\GGamma=(\Gamma_1,
\dots, \Gamma_m)$, listed 
left to right. We also keep track of the initial vertices $k_1, \ldots, k_m$ of these templates.

\begin{figure}[h]
\begin{center}
\includegraphics[height=2cm]{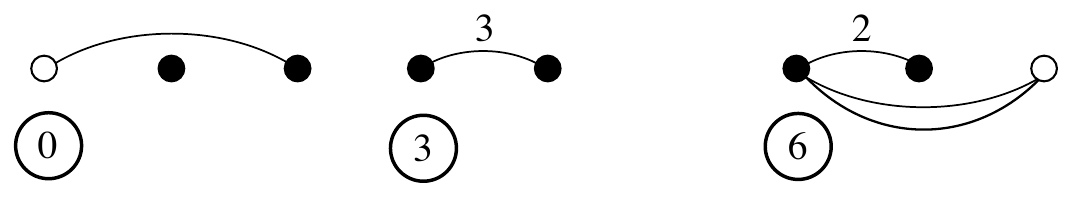}
\end{center}
\caption{The template decomposition of the floor diagram of Figure \ref{fig:bigfloordiagram}.}
\label{templates}
\end{figure}

Conversely, given the collection of templates $\GGamma=(\Gamma_1, \dots, \Gamma_m)$, the starting points $k_1, \ldots, k_m$, and the divergence sequence $\aa = (d^{\,t}, \rr - \ll)$, this process is easily reversed. To recover $\D'$, we first place the templates in their correct starting points in the interval $[1, \ldots, M]$, and draw in all the short edges that we removed from $\D'$ from left to right. More precisely, to change the divergences from $a_j(\GGamma)$ \footnote{We are denoting by $a_j(\GGamma)$ the divergence of vertex $j$ in the template $\Gamma_i$ containing it. Similarly, $\varkappa_j(\GGamma) = \varkappa_j(\Gamma_i) = a_0(\GGamma) + \cdots + a_j(\GGamma) = a_{k_i}(\Gamma_i) + \cdots + a_j(\Gamma_i)$.}
to $a_j$, we need to add $(a_0-a_0(\GGamma))+ \cdots + (a_{j-1}-a_{j-1}(\GGamma)) = (a_0 + \cdots + a_{j-1} - \varkappa_j(\GGamma))$ short edges between $j-1$ and $j$. Finally, we remove the first and last vertices and their incident edges to obtain $\D$.

Given a divergence sequence $\aa$ the possible starting points of the
templates in a collection $\GGamma = (\Gamma_1, \dots, \Gamma_m)$
 are restricted by $\aa$. More precisely, the valid 
sequences of starting points  $\kk = (k_1, \dots, k_m)$ of $\Gamma_1, \dots, \Gamma_m$ are
the ones in the set $A({\GGamma}, \aa)$ consisting of vectors $\kk
\in \NN^m$ such that
\smallskip
\begin{itemize}
\item $k_1 \ge 1 - \eps_0(\Gamma_1)$,
\smallskip
\item $ k_{i + 1} \ge k_i + l(\Gamma_i) \quad \mbox{for all } i = 1,
  \dots, m-1$, 
\smallskip
\item $k_m \le M - l(\Gamma_m) +
\eps_1(\Gamma_m)$, and
\smallskip
\item  $a_0 + \cdots + a_{k_i + j - 1} - \varkappa_j(\Gamma_i) \ge 0 \quad \mbox{for
} i = 1, \dots, m, \mbox{ and } j=1, \dots, l(\Gamma_i)$.
\end{itemize}
\smallskip
The first three inequalities guarantee that the templates fit in the
interval $[1, \ldots, M]$ without overlapping. The last condition
guarantees that the numbers of edges we need to add are
non-negative. Notice that, for fixed $\aa$, if $\delta(\D) = 0$ (\emph{i.e.}, if $\D$ is the
unique floor diagram with only short edges and $s_i = 0$ for $i \ge
2$) then $A(\GGamma, \aa)$ is empty as the decomposition removes all
edges. Due to this abnormality we exclude the case $\delta = 0$ in the sequel, though it is
not hard to see that $N^{\Delta, 0} = 1$ for all $\Delta$. 

We summarize the previous discussion in a proposition.

\begin{proposition}\label{prop:bijection}
Let $M \geq 1$, and let  $\ll, \rr \in \ZZ^M$, $\ss \in
\NN^M$. Let $d^{\,t} = s_1 + \cdots + s_M$ and $\aa = (d^{\,t},
\rr-\ll)$. The procedure of template decomposition is a bijection
between the $(\ll, \rr, \ss)$-floor diagrams and the pairs $(\GGamma, \kk)$ of a collection of templates~$\GGamma$ and a valid sequence of starting points $\kk \in A(\GGamma, \aa)$.
\end{proposition}

\subsection{Multiplicity, cogenus, and markings.}

Now we show that the multiplicity, cogenus, and markings of a floor diagram behave well under template decomposition.


\subsubsection{Multiplicity.}
If a floor diagram $\D$ has template decomposition $\GGamma$, then clearly
\begin{displaymath}
\mu(\D) = \prod_{i = 1}^m \mu(\Gamma_i).
\end{displaymath}

\subsubsection{Cogenus.}
Define the \emph{reversal sets} $\Rev(\rr)$ of the sequences
$\rr$ and $\ll$ by
\[
\Rev(\rr) = \{1 \le i < j \le M: r_i < r_j \}, \quad \Rev(\ll) = \{1 \le i < j \le M: l_i > l_j \}.
\]
The asymmetry is due to the fact that
the ``natural" order for $\rr$ is the weakly decreasing one, while 
for $\ll$ it is the weakly increasing one.
Define the \emph{cogenus} of the pair $(\ll, \rr)$ as
\[
\delta(\ll, \rr) = \sum_{(i,j) \in \Rev(\rr)} (r_j - r_i) \quad + \sum_{(i,j) \in \Rev(-\ll)} (l_i - l_j).
\]
Note that in the corresponding tropical curve, $(r_j - r_i) + (l_i - l_j)$ is the number of times that floors $i$ and $j$ cross, counted with multiplicity.

Given a collection of templates $\GGamma = (\Gamma_1, \dots,
\Gamma_m)$ we abbreviate the sum over their cogenera by
$\delta(\GGamma) := \sum_{i=1}^m \delta(\Gamma_i)$. 
 The template decomposition is cogenus
preserving, in the sense that
\begin{displaymath}
\delta(\D) = \delta(\GGamma) + \delta(\ll, \rr).
\end{displaymath}
This is because the tropical curve corresponding to $\D$ has
  $\delta(\D)$ nodes, counted with multiplicity; see
  Figure~\ref{fig:tropicalcurve}. The nodes arise in one of three
  ways: from
  an elevator crossing a floor, from an elevator with multiplicity greater than $1$, or from the crossing of two
  floors. There are exactly $\delta(\GGamma)$ tropical nodes
  of the first two kinds
  and  $\delta(\ll, \rr)$ of the last kind, each counted
  with multiplicity.

\subsubsection{Markings.}
The number of markings of a floor diagram is also expressible in terms of the ``number of markings of the templates''. The reason is simple: In Step~4 of Definition \ref{def:marking}, where we need to linearly order $\tilde{\D}$, we can linearly order each template independently. We need to introduce some notation.

Let $\D$ be a floor diagram with divergence sequence $\aa = (a_0, \ldots, a_M) = (d^{\,t}, \rr-\ll)$.
For each template $\Gamma$ and
each non-negative integer $k$ (for which (\ref{eq:numberofshortedges})
is non-negative for all $j$) let $\Gamma_{(\aa, k)}$ denote the graph
obtained from $\Gamma$ by first adding
\begin{equation}
\label{eq:numberofshortedges}
a_0 + a_1 + \cdots + a_{k+j-1} - \varkappa_j(\Gamma)
\end{equation}
short edges connecting $j-1$ to $j$, for $1 \le j \le l(\Gamma)$ (so that the vertices now have divergences $a_k, \ldots, a_{k+l(\Gamma)}$), and
then subdividing each edge of the resulting graph by introducing one
new vertex for each edge. Let $P_\Gamma(\aa, k)$ be the number of linear extensions (up to
equivalence) of the vertex poset of the graph $\Gamma_{(\aa, k)}$
extending the vertex order of $\Gamma$. Then 
\begin{displaymath}
\nu(\D) = \prod_{i=1}^m P_{\Gamma_i}(\aa, k_i) = : P_{\GGamma}(\aa, \kk).
\end{displaymath}

\subsection{Severi degrees in terms of templates.}

With this machinery the Severi degree $N^{\Delta, \delta}$ can be
computed solely  in terms of templates. We conclude from  Theorem~\ref{cor:keycorollary}, Proposition~\ref{prop:bijection}, and the previous observations in this section:

\begin{proposition} \label{prop:severitemplates}
For any $h$-transverse polygon $\Delta$ and $\delta \ge
  1$ the Severi degree $N^{\Delta, \delta}$  is given by
\begin{equation}
\label{eq:Severi2}
\tag{{\color{blue}Severi2}}
N^{\Delta, \delta} = \sum_{(\ll,\rr) \, : \, \delta(\ll,\rr) \leq \delta}\,\,\, \sum_{\GGamma \, : \, \delta(\GGamma) = \delta - \delta(\ll,\rr)} \,\,
\left( \prod_{i = 1}^m
\mu(\Gamma_i) \sum_{\kk \in A(\GGamma, \aa)}
 P_{\GGamma}(\aa, \kk) \right)
\end{equation}
where the first sum is over all permutations $\ll = (l_1, \dots, l_M)$
and $\rr = (r_1, \dots, r_M)$ of the left and right directions $D_l$
and $D_r$ of $\Delta$ with $\delta(\ll,\rr) \leq \delta$, and the second sum is over collections of templates $\GGamma$ of cogenus $\delta - \delta(\ll, \rr)$.
As before, we denote the upper edge
  length $d^{\,t}$ of $\Delta$ by
$a_0$, and write $a_i = r_i - l_i$ for $1 \le i \le M$.
\end{proposition}

(\ref{eq:Severi2}) improves (\ref{eq:Severi1}) by removing the unwieldy divergence condition on floor diagrams. However, eventual polynomiality is still far from clear.


\section{Polynomiality of Severi degrees: the ``first-quadrant" case} \label{sec:polynomialityspecial}

We will now use (\ref{eq:Severi2})
to prove our main theorem: the polynomiality of the Severi degrees for
toric surfaces given by sufficiently large $h$-transverse polygons. We
do this in two steps. In this section we carry out the proof in detail for the family of \emph{first-quadrant} polygons. 
The proof of this special case exhibits essentially all the features of the general case, and has the advantage of a more transparent notation. In Section \ref{sec:polynomialitygeneral} we explain how the arguments in this section are easily adapted to the general case.

In turn, we will first prove polynomiality of the Severi degrees for a fixed toric surface and variable multidegree (Theorem \ref{thm:fixed}). It will then be easy to extend this proof to also show polynomiality as a function of the surface (Theorem \ref{thm:universal}).

\begin{notation}
We say that an $h$-transverse polygon $\Delta=\Delta(\cc,\dd)$ is a \emph{first-quadrant polytope} if $\cc^l= \textbf{0} = (0, \ldots, 0)$ and $\cc^r \geq \textbf{0}$. We will then omit $\cc^l$ and $\dd^l$ from the notation and write $\Delta(\cc,\dd) =\Delta(\cc^r, (d^{\,t}; \dd^r)) = \Delta((c_1, \ldots, c_n), (d_0; d_1, \ldots, d_n))$. 
The corresponding floor diagrams have $M = d_1 + \cdots + d_n$ vertices. 
The multisets of left and right directions, and upper edge length are 
\begin{displaymath}
D_l = \{ 0, \dots, 0\}, \quad
D_r = \{\underbrace{c_1, \dots, c_1}_{d_1}, \dots, \underbrace{c_n, \dots, c_n}_{d_n} \}, \quad d_0.
\end{displaymath}Then $\ll=\textbf{0}$ and $\aa=(d_0, \rr)$. 
We write $\delta(\rr)=\delta(\ll, \rr)$. Notice the subtle distinction between $\aa$ and $\rr$, which will become more important in Section \ref{sec:polynomialitygeneral}.
\end{notation}

For example, Figure \ref{fig:polygon} shows the polygon $\Delta((1,0),
(1; 3,1))$ which has right directions $1$ and $0$ with respective
lengths $3$ and $1$, and upper edge length equal to 1. Here $D_r=\{1,1,1,0\}$. 

\begin{figure}[h]
\begin{center}
\qquad \includegraphics[height=3cm]{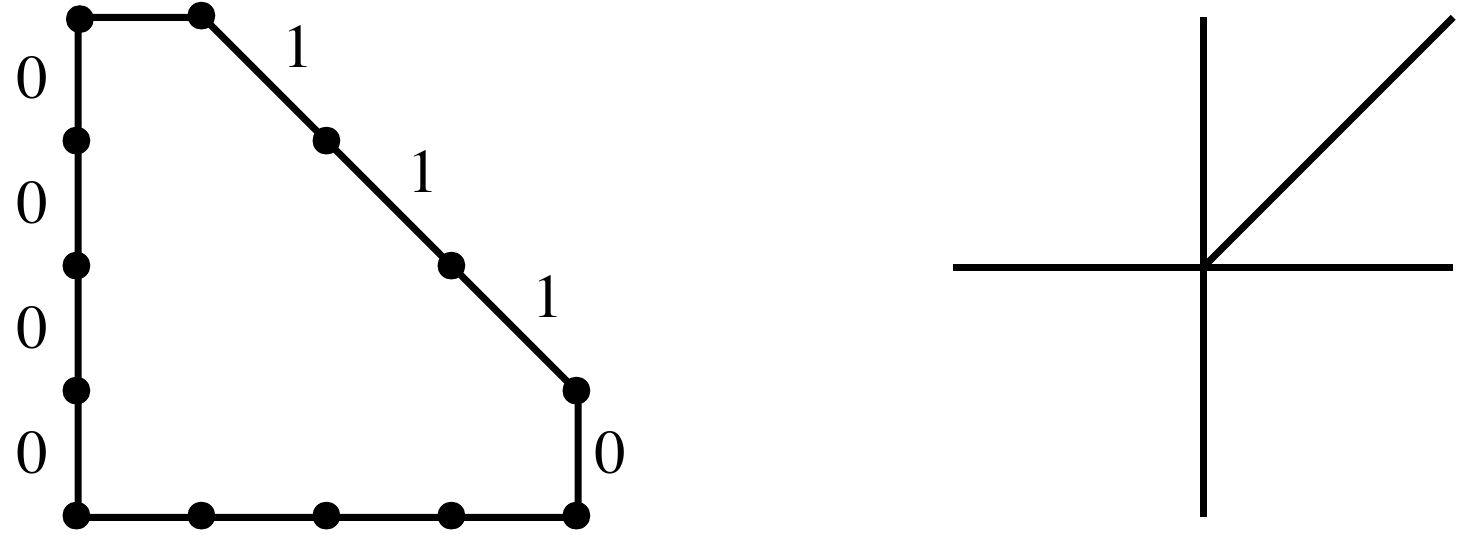}
\end{center}
\caption{The first-quadrant polygon $\Delta((1,0), (1; 3,1))$ and its normal fan.}
\label{fig:polygon}
\end{figure}

\begin{remark}
In this section we will assume that $\Delta(\cc,\dd)$ is a first-quadrant $h$-transverse polygon. We will also assume that
\[
d_0 \geq \delta, \quad d_0+c_1 \geq 2\delta, \quad d_1, \ldots, d_n \geq \delta+1 
\]
and will simply say that $\dd$ is \emph{large enough} to describe these inequalities. Throughout most of the section we we will hold $\cc$ constant and let $\dd$ vary. (When we let $\cc$ vary, we will say so explicitly.)
\end{remark}

(\ref{eq:Severi2}) now reads:
\begin{equation}
\label{eq:Severi2'}
N^{\Delta, \delta} = \sum_{\rr: \delta(\rr) \le \delta} \,\,\, 
\sum_{\GGamma: \,   \delta(\GGamma) = \delta -
  \delta(\rr)} \,\,\, \left( \prod_{i = 1}^m
\mu(\Gamma_i) \sum_{\kk \in A(\GGamma, \aa)} P_{\GGamma}(\aa, \kk) \right)
\tag{{\color{blue}Severi2'}}
\end{equation}


To show that (\ref{eq:Severi2'}) yields an eventual polynomial in $\cc$ and $\dd$, our first problem is that the index set of the first sum is hard to control: as $\cc$ and $\dd$ vary, the index set of permutations $\rr$ such that $\delta(\rr) \leq \delta$ varies quite delicately with them. In particular, these permutations can be arbitrarily long. In turn, the index set of the second sum depends very sensitively on the value of $\delta(\rr)$.
These problems are solved by presenting a more compact encoding of $\rr$.



\subsection{From permutations to swaps.}\label{sec:divtoswaps}
Let us organize the permutations $\rr$ of $D_r$ of cogenus less than or equal $\delta$ in a way which is uniform for large $\cc$ and $\dd$. 
Observe that, 
if $\dd$ is large enough, then such a permutation cannot contain a reversal of  $c_i$ and $c_j$ for $i \geq j+2$. This is because the minimum ``divergence cost" of reversing $c_{i-1}$ and $c_{i+1}$ is $d_i\min\{c_i-c_{i+1}, c_{i-1}-c_i\} \geq d_i > \delta$.


This observation allows us to encode such a permutation $\rr$ into $n-1$ sequences of $1$s and $-1$s which, for each $1 \leq i \leq n-1$, record the relative positions between the $c_i$s and the $c_{i+1}$s.

\begin{example} Suppose $\cc=(5,3,2,1)$, $\dd=(0,4,6,4,3)$, and $\rr=55335353233212121$. This permutation decomposes into three sequences of $1$s and $-1$s as follows:
\[
\begin{array}{ccccccccccccccccccc}
\aa & = & 5 &5 &3 &3&5&3&5&3&2&3&3&2&1&2&1&2&1\\
\\
1 & : & -1&-1&\1&\1&\m\1&\1&\m\1&1& &1&1& & & & & & \\
2 & : &  & &-1 &-1&&-1&&-1&\1&\m\1&\m\1&1&&1&&1&\\
3 & : &  & & &&&&&&-1&&&-1&\1&\m\1&\1&\m\1&1\\
\end{array}
\\
\]
\[
\pi_1=(1,1,-1,1,-1),\quad \pi_2 = (1,-1,-1), \quad \pi_3 = (1,-1,1,-1)
\]
\end{example}

To achieve uniformity among different sequence lengths, we delete all initial $-1$s and all final $1$s in each such sequence. The result is a \emph{swap}, which we define to be a sequence of $-1$s and $1$s which (is empty or) starts with a $1$ and ends with a $-1$.

We have encoded a permutation $\rr$ into a sequence of $n-1$ swaps $\pi=(\pi_1, \ldots, \pi_{n-1})$. Conversely, if we know $\cc$ and $\dd$ we can easily recover $\rr=:\pi(\cc,\dd)$ from $\pi$.

%
%

The following simple technical lemma will be crucial later, in the proof of Proposition \ref{prop:polytopesandpolynomials2}. 

\begin{lemma}
\label{lem:pwpoly}
Fix a collection ${\bf \pi} = (\pi_1, \dots,
\pi_{n-1})$ of swaps. Then, for
 $\cc = ( c_1 > \cdots
> c_n) \in \ZZ_{>0}^n$, $\dd = (d_0; d_1, \dots, d_n)\in \ZZ_{\geq 0}^{n+1}$ 
and $0 \le t \le \sum_{i = 1}^n
d_i$, the function (of $\cc$, $\dd$ and $t$)
\[
a_0 + a_1 + \cdots + a_t = d_0
 + \pi(\cc,\dd)_1 + \cdots + \pi(\cc,\dd)_t
\]
is piecewise polynomial in $\cc$, $\dd$ and $t$ for large enough $\dd$. Here $\aa = (d_0; \rr) = (d_0, \pi(\cc,\dd))$. The regions of polynomiality are the faces of a hyperplane arrangement.
%
\end{lemma}

\begin{proof}
Say $\pi_i$ contains $\alpha_i$ $1$s and $\beta_i$ $-1$s. We claim that, for large $\dd$, the function $a_0+\cdots+a_t$ is polynomial when restricted to the lattice points in a fixed face of the following hyperplane arrangement in $(\cc,\dd,t)$-space:
\begin{displaymath}
t = d_1 + \cdots + d_i + f  \qquad \textrm{ for } -\beta_i \leq f \leq \alpha_i \quad (1 \leq i \leq n)
\end{displaymath}
This is easy to see because 
\[
\aa=(d_0, \pi(\cc,\dd)) = 
(d_0, \underbrace{c_1, \ldots, c_1}_{d_1-\beta_1},
\underbrace{c_1\textrm{s and } c_2\textrm{s}}_{\alpha_1+\beta_1}, \underbrace{c_2, \ldots, c_2}_{d_2-\alpha_1-\beta_2},
\underbrace{c_2\textrm{s and } c_3\textrm{s}}_{\alpha_2+\beta_2}, \underbrace{c_3, \ldots, c_3}_{d_3-\alpha_2-\beta_3},\ldots)
\]
where the order of the $c_i$s and $c_{i+1}$s is determined by $\pi$. Thus $a_0 + a_1 + \cdots + a_t$ equals
\[
\begin{cases}
d_0+tc_1 & \textrm{ if } 0 \leq t \leq d_1-\beta_1 \\
\quad d_0+tc_1+* & \quad \textrm{ if } t = d_1-\beta_1+1 \\
\quad \quad \vdots & \quad \quad \vdots \\
\quad d_0+tc_1+*+\cdots+* & \quad \textrm{ if } t= d_1+\alpha_1 -1\\
d_0+d_1c_1+(t-d_1)c_2 & \textrm{ if } d_1+\alpha_1 \leq t \leq d_1+d_2-\beta_2\\
\quad d_0+d_1c_1+(t-d_1)c_2+* & \quad \textrm{ if } t = d_1+d_2-\beta_2+1\\
\quad \quad \vdots & \quad \quad \vdots \\
\quad d_0+d_1c_1+(t-d_1)c_2+*+\cdots+* &\quad  \textrm{ if } t=  d_1+d_2+\alpha_1 -1\\
d_0+d_1c_1+d_2c_2+(t-d_1-d_2)c_3 & \textrm{ if } d_1+d_2+\alpha_1 \leq t \leq d_1+d_2+d_3-\beta_3 \\
\quad \vdots & \quad \vdots 
\end{cases}
\]
where each $*$ represents a $c_i$ determined by $\pi$. The claim follows.
%
\end{proof}

Using the encoding of permutations into swaps, we now replace the first sum in
(\ref{eq:Severi2'}) by a sum over swaps. 
Let the \emph{number of inversions} $\inv(\pi)$ of a swap $\pi$ be
\[
\inv(\pi) = \# \{(i,j) \in \ZZ^2 : 1 \leq i < j \leq n-1 \text{ and } \pi(i) > \pi(j) \}.
\]
It is easy to see that $\delta(\rr) = \sum_{i = 1}^{n-1} \inv(\pi_i)
(c_{i} - c_{i + 1})$. We obtain that, for large $\dd$, 
\begin{equation}
\tag{{\color{blue}Severi3'}}
\label{eq:Severi3}
N^{\Delta, \delta} = \sum_{{\bf \pi} } \, \,
\sum_{\GGamma: \, \, \delta(\GGamma) = \delta -
  \delta(\rr)} \quad \left( \prod_{i = 1}^m
\mu(\Gamma_i)  \sum_{\kk \in A(\GGamma, \aa) \cap \ZZ^m} 
P_{\GGamma}(\aa, \kk) \right)
\end{equation}
where the first sum is now over all sequences ${\bf \pi} = (\pi_1, \dots, \pi_{n-1})$ of swaps with $\sum_{i = 1}^{n-1} \inv(\pi_i)
(c_{i} - c_{i + 1}) \le \delta$,  $\rr = {\bf \pi}(\cc, \dd)$, $\aa=(d_0, \rr)$ and the other sums are as before.


\medskip

For fixed $\cc$, the first sum in (\ref{eq:Severi3}) is finite and its index set is independent of $\dd$. Also, for each $\pi$ in that index set, $\delta(\rr)$ is independent of $\dd$, and hence so is the set of templates $\GGamma$ in the second sum. The difficulty encountered in (\ref{eq:Severi2'}) is resolved. 

If $\cc$ is variable this observation still applies, under the additional assumption that $\cc$ grows quickly enough that $c_i - c_{i+1} > \delta$ for all $i$. In that case, the first sum will only include the trivial swap sequence $\pi$ where every swap is empty, and then the index set of the second sum will still be independent of $\dd$, and also of $\cc$.
%

\medskip

In (\ref{eq:Severi3}) we have expressed $N^{\Delta, \delta}$ as a weighted sum of finitely many contributions of the form
\[
N^{\Delta, \delta}_{\pi, \GGamma} := 
 \sum_{\kk \in A(\GGamma, \aa) \cap \ZZ^m} P_{\GGamma}(\aa, \kk), 
\]
where $\aa = (d_0, \pi(\cc,\dd))$. Our final goal is to show that, for fixed $\delta, \GGamma$, and $\pi$, and for large $\dd$, this function varies piecewise polynomially in $\cc$ and $\dd$. 
We will do it over the course of the Sections \ref{sec:polytopality} -- \ref{sec:polySeveri} by showing that $A(\GGamma, \aa)$ is a variable polytope and $P_{\GGamma}(\aa, \kk)$ is piecewise polynomial, and then recurring to some facts about such \emph{discrete integrals}.


\subsection{Polytopality of $A(\Gamma,\aa)$.}\label{sec:polytopality}

Our next key proposition states that, for large enough $\cc$ and $\dd$, the innermost index set $A(\Gamma,\aa) \cap \ZZ^m$ of (\ref{eq:Severi3}) is the set of lattice points in a polytope. While it does vary as a function of $\dd$, it does so in a controlled way. 


\begin{proposition}
\label{prop:polytopesandpolynomials}
Let $\pi=(\pi_1, \ldots, \pi_{n-1})$ be a fixed sequence of swaps and
let \,\, $\GGamma = (\Gamma_1, \dots, \Gamma_m)$ be a fixed collection of
templates. Let $(c_1 > \cdots > c_n) \in \ZZ_{\geq 0}^n$ and $\dd =
(d_0; d_1, \dots, d_n) \in \ZZ_{\ge 0}^{n+1}$ be variable and assume
that $d_0 \geq \delta(\GGamma)$, $d_0+c_1 \geq 2\delta(\GGamma)$. Let $\aa=(d_0, \rr) = (d_0, \pi(\cc,\dd))$. 
Then
 $A(\GGamma, \aa)$ is the set of lattice points in a polytope whose facet directions are fixed, and whose facet parameters are linear functions of 
$d_0, \dots, d_n$.
\end{proposition}

\begin{proof}
The only non-linear conditions defining  $A(\GGamma, \aa)$  are the inequalities
\[
a_0 + \cdots + a_{k_i + j - 1} - \varkappa_j(\Gamma_i) \ge 0.
\]
for  $i = 1, \dots, m$ and $ j=1, \dots, l(\Gamma_i)$. We will show
that, under these hypotheses, they hold ``for free". In fact, we will
show the stronger statement $a_0 + a_1 - \varkappa_j(\Gamma_i) \ge 0$.
This implies all other inequalities (since all $a_j$ are non-negative)
with the possible exception of $a_0 - \varkappa_1(\Gamma_1) \geq 0$ if $k_1=0, j=1$, which we deal with separately.

For every $i = 1, \dots, m$ we have 
\[
\varkappa_j(\Gamma_i) \le \sum_{\stackrel{e}{r \to s}} w(e) \le 2 \sum_{\stackrel{e}{r \to s}} \bigg[(s-r) w(e) -1
\bigg] = 2\delta(\Gamma_i) \le 2 \delta(\GGamma),
\]
where the sums are over the edges of $\Gamma_i$.
Therefore $\varkappa_j(\Gamma_i) \le d_0 + 
c_1 - 2\delta(\rr) $.  Finally, from the definition of the cogenus
$\delta(\rr)$ 
it is clear that $\delta(\rr) \ge c_1 -
r_1$. Therefore $\varkappa_j(\Gamma_i) \le d_0+r_1 - \delta(\rr) \leq a_0 + a_1 $ as desired, since $a_0=d_0$ and $a_1=r_1$.

Now we prove $a_0 - \varkappa_1(\Gamma_1) \geq 0$ for $k_1=0$ and $j=1$. 
If $k_1=0$ we must have $\eps_0(\Gamma_1) = 1$, so all
edges of $\Gamma_1$ adjacent to vertex $0$ have weight
$1$ and there are no edges between $0$ and $1$. Thus, we have $\varkappa_1(\Gamma_1) \le \delta(\Gamma_1) \le
\delta(\GGamma)$, and from our assumption $\delta(\GGamma) \le d_0$ we
conclude that $\varkappa_1(\Gamma_1) \le a_0$.

Finally notice that the facet directions are fixed, and the only non-constant facet parameter depends only on $M = d_1 + \cdots + d_n$.
\end{proof}

\begin{remark}
In a sense, the ``real content" of the previous proof is the following statement: when $d_0 \ge \delta$ and
$d_0 + c_1 \ge 2\delta$,
all templates can move between the
first and last vertex of the floor diagram
without obstruction.
\end{remark}

 \subsection{Piecewise polynomiality of $P_{\GGamma}(\aa,\kk)$.} \label{sec:PP}

%

\bigskip

\begin{proposition}
\label{prop:polytopesandpolynomials2}
Let  $\pi = (\pi_1, \ldots, \pi_{n-1})$ be a fixed collection of swaps and let \,\, $\GGamma =
(\Gamma_1, \dots, \Gamma_m)$ be a fixed collection of templates. Let
$\cc = (c_1 > \cdots > c_n) \in \ZZ_{>0}^n$ and $\dd = (d_0; d_1, \dots, d_n) \in \ZZ_{\ge 0}^{n+1}$ be variable and $\aa=(d_0, \rr) = (d_0, \pi(\cc,\dd))$. Let $\kk \in A(\GGamma, \aa) \cap \ZZ^m$ be variable.
Then the
  function $P_{\GGamma}(\aa,\kk)$ 
  is piecewise
  polynomial in $\cc$, $\dd$ and $\kk$.
  The domains of polynomiality are faces of a hyperplane arrangement.
\end{proposition}

\begin{proof}
Recall that $P_{\GGamma}(\aa,\kk) = \prod_{i=1}^m P_{\Gamma_i}( \aa, k_i)$. 
Let $\Gamma=\Gamma_i $ be one of the templates in $\GGamma$ and let $k=k_i$.
By definition, $P_\Gamma(\aa, k)$ is the number of linear extensions of the acyclic graph $\Gamma_{(\aa, k)}$ extending the order of the template $\Gamma$. Recall how this graph $\Gamma_{(\aa, k)}$ is obtained from $\Gamma$: we add in the right number of short edges to $\Gamma$ (more precisely, $a_0 + \cdots + a_{k+j-1} -
  \varkappa_{j}$ edges between vertices $j-1$ and $j$) so that the resulting graph has divergences $a_k, \ldots, a_{k+l(\Gamma)-1}$, and then we introduce a new vertex at the midpoint of each edge.

Such a linear extension on $\Gamma_{(\aa, k)}$ can be constructed in
two steps. In Step 1, we choose a linear order (modulo equivalence) of
the graph formed by the vertices $0, \dots, l(\Gamma)$ of $\Gamma$ and
the midpoint vertices coming from edges of $\Gamma$. In Step 2 we
insert the midpoint vertices of the new edges of $\Gamma_{(\aa, k)}$
into the linear order of Step~1. If $b_j$ is the number of vertices
between $j-1$ and $j$ in the linear order of Step~1, there are
\begin{equation}
\label{eq:numberofextensions}
\prod_{j=1}^l 
\binom{a_0 + \cdots + a_{k+j-1} -
  \varkappa_{j}(\Gamma) + b_j}{b_j}
\end{equation}
ways to insert those midpoints, up to equivalence. 

Notice that the parameters $\varkappa_{j}(\Gamma)$ and $b_j$ are constants that depend only on $\Gamma$.
Lemma \ref{lem:pwpoly} tells us that $a_0 + \cdots + a_{k+j-1}$ is a
piecewise polynomial in $\cc$, $\dd$ and $k$, and the proof describes
the domains of polynomiality. This allows us to conclude that the
expression of (\ref{eq:numberofextensions}) is polynomial on each face of the hyperplane arrangement
\[
k = d_1 + \cdots + d_i + f, \qquad  -\beta_i - l(\Gamma) + 1 \leq f \leq \alpha_i, 
\]
for $1 \leq i \leq n$,
and thus $P_{\GGamma}(\aa,\kk)$ is polynomial on each face of the following arrangement $\A$ in $(\cc,\dd,\kk)$-space:
\[
\A: \quad k_s = d_1 + \cdots + d_i + f, \qquad  -\beta_i - l(\Gamma_s) + 1 \leq f \leq \alpha_i, 
\]
for $1 \leq i \leq n$ and $1\leq s \leq m$.
\end{proof}


\medskip

\subsection{Discrete integrals of polynomials over polytopes.}\label{sec:discint}

In Section \ref{sec:polytopality} we showed that $A(\GGamma, \aa) \cap \ZZ^m$ is the set of lattice points in a polytope with fixed facet directions, and whose facet parameters are linear functions of $\dd$. Since this set only depends on~$\dd$, we relabel it $A(\GGamma, \dd) \cap \ZZ^m$.
In Section \ref{sec:PP} we showed that $P_{\GGamma}(\aa, \kk)$ is a piecewise polynomial function of $\cc, \dd,$ and $\kk$, whose domains of polynomiality are cut out by a hyperplane arrangement $\mathcal A$. The equations of this arrangement have fixed normal directions, and parameters which are linear functions of $\dd$ and $\kk$. It follows that 
\[
N^{\Delta, \delta}_{\pi, \GGamma}= 
\sum_F  \sum_{\kk \in (A(\GGamma, \dd) \cap F)^o \cap \ZZ^m} P^F_{\GGamma}(\cc, \dd, \kk), 
\]
summing over the faces $F$ of $\mathcal A$, where each $P_\GGamma^F$ is a polynomial. Here $Q^o$ denotes the relative interior of $Q$, \emph{i.e.}, the interior of $Q$ with respect to its affine span. We get:
\begin{equation}
\tag{{\color{blue}Severi4}}
\label{eq:Severi4}
N^{\Delta, \delta}= 
\sum_{\pi, \GGamma, F} \,\, \sum_{\kk \in (A(\GGamma, \dd) \cap F)^o \cap \ZZ^m} P^F_{\GGamma}(\cc, \dd, \kk).
\end{equation}
This is a somewhat messy expression, but the point is that there is a finite number of choices for $\pi, \GGamma,$ and $F$, and these choices are independent of $\dd$. Now we just need to prove the polynomiality of the inner sum, which is a discrete integral of a polynomial function over a variable open polytope.

\medskip

To do so, we invoke some results on discrete integrals. Given a polytope $Q \subset \RR^m$ and a function $f:\RR^m \rightarrow \RR$, we define the \emph{discrete integral} of $f$ over $Q$ to be
\[
\sum_{q \in Q \cap \ZZ^m} f(q).
\]

Recall that an $m$-polytope is \emph{simple} if every vertex is contained in exactly $m$ edges. It is \emph{integral} if all its vertices have integer coordinates. 
A \emph{facet translation} of a polytope $P=\Pi_X(\yy) = \{\kk \in \RR^m \, : \, X\kk \leq \yy\}$ is a polytope of the form $\Pi_X(\yy') = \{\kk \in \RR^m \, : \, X\kk \leq \yy'\}$ for $\yy' \in \RR^l$, obtained by translating the facets of $P$. 
We assume that $X$ is an integer matrix and 
say $\Pi_X(\yy')$ is an \emph{integer facet translation} if $\yy' \in \ZZ^l$. Say that the matrix $X$ is \emph{unimodular}, and that $P$ is \emph{facet-unimodular}, if every maximal minor has determinant $-1, 0,$ or $1$.  When this is the case, every integer facet translation $\Pi_X(\yy')$ has integral vertices by Cramer's rule.

The values of $\yy'$ for which $\Pi_X(\yy')$ and $P$ are combinatorially equivalent form an open cone in $\RR^m$; its closure is the \emph{deformation cone} of $P$. The corresponding polytopes are called \emph{deformations} of $P$. \cite{P05, PRW}

Recall that a \emph{quasipolynomial} function on a lattice $\Lambda$
is a function which is polynomial on each coset of some finite index sublattice $\Lambda' \subseteq \Lambda$. Results like the following are known, although we have not found in the literature the precise statement that we need:

\begin{lemma}\label{lemma:discreteintegral1}
Consider the integer facet translations of a simple rational polytope $P$ with fixed facet directions and variable facet parameters, \emph{i.e.}, the polytopes 
\[
\Pi_X(\yy) = \{\kk \in \RR^m \, : \, X\kk \leq \yy\},
\]
where $X \in \ZZ^{l \times m}$ is a fixed $l \times m$ matrix and $\yy \in \ZZ^l$ is a variable vector.
Let $f$ be a polynomial function and let
\[
g(\yy) = \sum_{\kk \, \in \, \Pi_X(\yy) \cap\,  \ZZ^m} f(\kk), \qquad g^o(\yy) = \sum_{\kk \, \in \, \Pi_X(\yy)^o \cap\,  \ZZ^m} f(\kk).
\]
be the discrete integrals of $f$ over $\Pi_X(\yy)$, and over its relative interior.
Then $g(\yy)$ and $g^o(\yy)$ are piecewise quasipolynomial functions of $\yy$. The domains of quasipolynomiality are given by linear conditions in $\yy$. More concretely, these functions are quasipolynomial when restricted to those $\yy$ for which the polytope $\Pi_X(\yy)$ has a fixed combinatorial type. 
\\
Furthermore, if $X$ is unimodular, then $g(\yy)$ and $g^o(\yy)$ are piecewise polynomial.\end{lemma}

\begin{proof}
This is certainly known for $f=1$, \emph{i.e.}, for the lattice point count $g(\yy)=|\Pi_X(\yy) \cap\,  \ZZ^m|$. For instance, a proof can be found in \cite[Theorem 19.3]{P05} for the parameters $\yy$ for which $\Pi_X(\yy)$ is integral. This proves the unimodular case, and is easily adapted to the non-unimodular case. That proof is easily modified to apply to any polynomial $f$. By subtracting off the boundary faces of our polytope (with alternating signs depending on the dimension) we obtain the results for $g^o$.
\end{proof}

\begin{lemma}\label{lemma:discreteintegral2}
 Consider a variable polytope with fixed facet directions, and facet parameters which vary linearly as a function of a vector $\dd$; \emph{i.e.}, 
\[
\Pi_X(Y\dd) = \{\kk \in \RR^m \, : \, X\kk \leq Y\dd\}
\]
where $X \in \ZZ^{l \times m}$ and $Y \in \ZZ^{l \times n}$ are fixed $l \times m$ and $l \times n$ matrices, and $\dd \in \RR^n$ is a variable vector.
Let $f(\cc, \dd, \kk)$ be a polynomial function of $\cc \in \RR^n$, $\dd \in \RR^n$, and $\kk \in \RR^m$, and let 
\[
g(\cc,\dd) = \sum_{\kk \, \in \, \Pi_X(\yy) \cap\,  \ZZ^m} f(\cc,\dd,\kk), \qquad
g^o(\cc,\dd) = \sum_{\kk \, \in \, \Pi_X(\yy)^o \cap\,  \ZZ^m} f(\cc,\dd,\kk).
\]
Then $g(\cc,\dd)$ and $g^o(\cc,\dd)$ are piecewise polynomial functions of $\cc$ and $\dd$. The domains of quasipolynomiality are given by linear conditions in $\dd$. More concretely, these functions are quasipolynomial when restricted to those $\dd$ for which the polytope $\Pi_X(Y\dd)$ has a fixed combinatorial type. \\
Furthermore, if $X$ is unimodular, then $g(\cc,\dd)$ and $g^o(\cc,\dd)$ are piecewise polynomial.
\end{lemma}

\begin{proof}
This is an easy consequence of the previous lemma. 
Write $f(\cc, \dd, \kk) = \sum_\ii f_\ii(\cc, \kk)\dd^\ii$. By Lemma \ref{lemma:discreteintegral1},  $\sum_{\kk \, \in \, \Pi_X(\yy) \cap\,  \ZZ^m} f_\ii(\cc,\kk)$ is a piecewise polynomial in $Y\dd$, and therefore in $\dd$, with polynomials in $\cc$ as coefficients. The domains of polynomiality are given by linear conditions in $Y\dd$, and hence in $\dd$. Now sum over all $\ii$ to obtain the desired result.
\end{proof}

\subsection{Polynomiality of Severi degrees.}\label{sec:polySeveri}

We are now ready to prove the eventual polynomiality of Severi degrees in the special case of first-quadrant polygons.

\begin{theorem} (Polynomiality of first-quadrant Severi degrees 1: Fixed Surface.)\label{thm:fixed}\\
\noindent Fix $n \ge 1$, $\delta \ge 1$, and $\cc=(c_1> \cdots > c_n) \in \ZZ^n$. There is a polynomial
$p^\cc_\delta(\dd)$ such that the Severi degree $N_{\Tor(\cc)}^{\dd, \delta}$ is
  given by
\begin{equation}
\label{eqn:polynomiality}
N_{\Tor(\cc)}^{\dd, \delta} = p^\cc_\delta(\dd)
\end{equation}
for any $\dd \in
\ZZ_{\ge 0}^{n+1}$ such that $d_0 \geq \delta$, $d_0+c_1 \geq 2\delta$ and $d_1, \ldots, d_n \geq \delta+1$.
\label{thm:maintheorem}

\end{theorem}

\begin{proof}
We do this in three steps. 

\medskip

\noindent\emph{Step 1. Piecewise quasipolynomiality}. 
In (\ref{eq:Severi4})
there is a fixed (and finite) set of choices for $\pi, \GGamma$, and $F$, independently of $\dd$. For each such choice, the function $P^F_{\GGamma}(\cc, \dd, \kk)$ is polynomial in $\cc, \dd$, and $\kk$ (thanks to Section \ref{sec:PP}) and the domain $A(\GGamma, \dd) \cap F$ is polytopal with fixed facet directions and facet parameters which are linear in $\dd$ (thanks to  Section \ref{sec:polytopality}). Lemma \ref{lemma:discreteintegral2} then shows that $N^{\Delta, \delta}$ is piecewise quasipolynomial in $\cc$ (which is constant here) and $\dd$.

\medskip

\noindent\emph{Step 2. Quasipolynomiality.}
To prove that all large $\dd$ lie in the same domain of quasipolynomiality, we need to analyze those domains more carefully. Each polytope $A(\GGamma, \dd) \cap F$ is the space of $(k_1, \ldots, k_m)$ such that
\[
\begin{array}{rcll}
k_s &\bigcirc & d_1 + \cdots + d_i + f  \quad & (-\beta_i - l_s + 1 \leq f \leq \alpha_i, \,\, 1 \leq i \leq n, \,\,1\leq s \leq m) \\
k_1 &\ge&  1 - \eps_0 & \\
 k_s - k_{s-1} &\ge&  l_{s-1} &(2 \leq s \leq m) \\
 k_m &\le & M - l_m + \eps_1, &
\end{array}
\]
where $\bigcirc$ represents $\geq, =$, or $\leq$, and we abbreviate $l_i:= l(\Gamma_i), \eps_0:= \eps_0(\Gamma_1)$ and $\eps_1:= \eps_1(\Gamma_m)$. We need to show that the combinatorial type of this polytope does not depend on $\dd$.

Let's examine how the parameters in $\dd$ restrict the positions of the integers in $\kk$ when $\dd$ is large. The numbers $d_1, d_1+d_2, \ldots, d_1+\cdots+d_n=M$ are far from each other. The first set of inequalities ``anchor" some of the $k_s$s  to be very near the number $d_1+\cdots+d_i$. If $k_s$ is anchored near $d_1+\cdots+d_i$, then it is forced to equal $d_1+\cdots+d_i+f$, for some $f \in [-\beta_i-l_s+1, \alpha_i]$ which is determined by $F$ independently of $\dd$. If $k_s$ is not anchored to any $d_1+\cdots+d_i$, then those inequalities allow it to roam freely inside one concrete large interval $[d_1+\cdots+d_i, d_1+\cdots+d_{i+1}]$, but not too close to either endpoint of the interval.

Since $\dd$ is large, the inequality $k_s-k_{s-1} \geq l_{s-1}$ is automatically satisfied by $\kk$ unless one of three things happen:
\begin{itemize}
\item
$k_{s-1}$ and $k_s$ are anchored to the same $d_1 + \ldots + d_i$,
\item
neither is anchored, and both are restricted to lie in the same interval.
\item
one of them is anchored to $d_1 + \cdots + d_i$, and the other one is restricted to one of the intervals adjacent to the same $d_1 + \cdots + d_i$.
\end{itemize}
In the first case, either the inequality $k_s-k_{s-1} \geq l_{s-1}$ automatically holds (and does not define a facet of $A(\GGamma, \dd)$) or it automatically does not hold (and the polytope is empty), depending on how far $k_{s-1}$ and $k_s$ are anchored from $d_1+\cdots+d_i$. 
In the second case, the inequality does not hold automatically, and therefore defines a facet  of $A(\GGamma, \dd)$. 
In the third case, the inequality may hold automatically (and not give a facet) or introduce a new restriction on $\kk$ (and give a facet); but again, this depends only on the anchoring, and is independent of $\dd$. 
A similar analysis holds for the inequalities $k_1 \ge  1 - \eps_0$ and $ k_m \le M - l_m + \eps_1$.

In summary, for large $\dd$, the ``shape" of the restrictions on $\kk$ (\emph{i.e.} the combinatorial type of $A(\GGamma, \dd) \cap F$) is independent of $\dd$. This proves that $N^{\Delta,\delta}$ is quasipolynomial for large $\dd$.

Now we discuss the restrictions on $\dd$ necessary for the previous analysis to hold. First, we need it to be impossible for $k_s$ to be anchored to $d_1+\cdots+d_i$ and to $d_1+\cdots+d_i+d_{i+1}$ simultaneously. This translates to  $d_1+\cdots+d_{i+1}-\beta_{i+1}-l_s+1 > d_1+\cdots+d_i+\alpha_i$, or $d_{i+1} \geq l(\pi_{i+1}) + l_s$, for all $s$.
We also need that, if $k_{s-1}$ is anchored to $d_1+\cdots + d_i$ and $k_s$ is anchored to $d_1 + \cdots + d_{i+1}$, we automatically have $k_s-k_{s-1} \geq l_{s-1}.$
This requires the inequality $d_1+\cdots+d_{i+1}-\beta_{i+1}-l_s+1 \geq d_1+\cdots+d_i+\alpha_i+l_s$, or $d_{i+1} \geq l(\pi_{i+1}) + l_s+l_{s-1}-1$, which is stronger than the previous one. This last inequality follows from two easy observations: $l(\pi) \leq \inv(\pi)+1$ for any swap $\pi$, and $l(\Gamma) \leq \delta(\Gamma)+1$ for all templates $\Gamma$. From these, and the assumption that $\dd$ is large, we get 
\[
d_{i+1} \geq \delta+1 = \delta(\rr)+ \delta(\Gamma)+1 \geq \inv(\pi) + l(\Gamma) \geq  l(\pi) + l_s+l_{s-1} -1
\] 
as desired.

\medskip

\noindent\emph{Step 3. Polynomiality.}
Finally, to prove polynomiality, we prove that the polytopes $A(\GGamma, \dd) \cap F$ are facet-unimodular. This is easy since the rows of the matrix describing this polytope are of the form $\ee_i$ or $\ee_i-\ee_j$, where $\ee_i$ is the $i$th unit vector. This is a submatrix of the matrix of the root system $A_m=\{\ee_i - \ee_j, \, 1 \leq i \neq j \leq m+1\}$, which is totally unimodular; \emph{i.e.}, all of its square submatrices have determinant $-1, 0$, or $1$. \cite{S82}
\end{proof}

\begin{theorem} (Polynomiality of first-quadrant Severi degrees 2: Universality.)\label{thm:universal}\\
\noindent Fix $n \ge 1$ and $\delta \ge 1$. There is a universal polynomial
$p_\delta(\cc,\dd)$ such that the Severi degree $N_{\Tor(\cc)}^{\dd, \delta}$ is
  given by
\begin{equation}
\label{eqn:polynomiality}
N_{\Tor(\cc)}^{\dd, \delta} = p_\delta(\cc,\dd).
\end{equation}
for any $\cc=(c_1 > \cdots > c_n) \in \ZZ^n$ and $\dd  \in
\ZZ_{\ge 0}^{n+1}$ such that $c_i - c_{i+1} \geq \delta+1$, $d_i \geq
\delta+1$ for all $i$, $d_0 \geq \delta$ and $d_0+c_1 \geq 2\delta$.
\label{thm:maintheoremuniversal}
\end{theorem}

\begin{proof}
We have already done all the hard work, and this result follows immediately from the discussion at the end of Section \ref{sec:divtoswaps}. If $c_i-c_{i+1} > \delta$ for all $i$ then $\delta(\rr) > \delta$ for any $\pi$ other than the trivial collection of empty swaps. Therefore, in this case (\ref{eq:Severi3}) says 
\[
N^{\Delta, \delta} = \sum_{\GGamma: \, \, \delta(\GGamma) = \delta} \quad \left( \prod_{i = 1}^m
\mu(\Gamma_i)  \sum_{\kk \in A(\GGamma, \aa) \cap \ZZ^m} 
P_{\GGamma}(\aa, \kk) \right)
\]
The indexing set for this sum no longer depends on $\cc$, so this is simply a weighted sum of functions which are polynomial in $\cc$ and $\dd$ when $\dd$ is large. The desired result follows.
\end{proof}
%
%
%
%

%

\begin{remark}
\label{rmk:algorithm}
This description gives, in principle, an explicit
algorithm to compute the polynomial $p_\delta(\cc, \dd)$. In Section 3
of \cite{FB}, the second author describes an algorithm which generates all templates of a given cogenus. The discrete integral 
\[
\sum_{\kk \in A(\GGamma, \aa) \cap \ZZ^m} P_{\GGamma}(\aa, \kk)
\]
can be evaluated symbolically by repeated application of Faulhaber's
formula (\cite[Lemma~3.5]{FB}, taken from~\cite{Kn93}).
\end{remark}

\section{Polynomiality of Severi degrees: the general $h$-transverse case} \label{sec:polynomialitygeneral}

We are now ready to prove our main results, Theorems \ref{mainthm:fixed} and \ref{mainthm:universal}, which assert the eventual polynomiality of the Severi degrees $N_{\Tor(\cc)}^{\dd, \delta}$ for arbitrary $h$-transverse polygons.
We simply adapt the proofs of Theorems \ref{thm:fixed} and \ref{thm:universal}  for first-quadrant polygons. The adaptation is fairly straightforward, though the details are slightly more cumbersome. 

\begin{remark}
In this section we assume that the $h$-transverse polygon $\Delta(\cc,\dd)$ with $\cc=(\cc^r;  \cc^l)$ and $\dd=(d^{\,t}; \dd^r; \dd^l)$ satisfies:
\[
d^{\,t}, d^{\,b} \geq \delta, \quad d^{\,t}+c^r_1-c^l_1, d^{\,b}+c^r_n-c^l_m \geq 2\delta, \quad d^{\,r}_1, \ldots, d^{\,r}_n, d^{\,l}_1, \ldots, d^{\,l}_m \geq \delta+1, 
\]
and
\[
\left|(d^{\,r}_1+\ldots+d^{\,r}_i) - 
(d^{\,l}_1+\ldots+d^{\,l}_j)\right| \geq \delta+2 \quad \textrm{ for } 1 \leq i \leq n-1, \,\, 1 \leq j \leq m-1
\]
\end{remark}

\begin{proof}[Proof of Theorem \ref{mainthm:fixed}]
We follow the steps of Section \ref{sec:polynomialityspecial} one at a time.

\noindent \emph{1. Swaps.}
The encoding of divergence sequences in terms of swaps still works, since $\dd^r$ and $\dd^l$ are large enough. Now we obtain two swap sequences $\pi^r$ and $\pi^l$ for $\rr$ and~$\ll$, respectively. Here $\aa=(d^{\,t}, \rr-\ll)$ where $\rr-\ll = \pi^r(\cc^r,\dd^r) - \pi^l(\cc^l,\dd^l)$, and the expression $a_0+\ldots+a_t$ is still piecewise polynomial in $\cc, \dd$ and $t$. The regions of polynomiality are given by how far $t$ is from the numbers $d^{\,r}_1+\cdots+d^{\,r}_i$ (as before) and $d^{\,l}_1+\cdots+d^{\,l}_j$.
\medskip

\noindent \emph{2. Polytopality.}
The domain $A(\GGamma, \aa)$ of possible template starting points is
still polytopal. To see this, once again, we prove that the
potentially non-linear inequalities $a_0+\cdots+a_{{k_i}+j-1} \geq
\varkappa_j(\Gamma_i)$ hold automatically, by proving that: 

$\bullet$ $a_0+\cdots+a_m \geq 2\delta(\GGamma)$, for $1 \le m \le M - 1$ (and recalling that $2\delta(\GGamma) \geq \varkappa_j(\Gamma_i$)),

$\bullet$ $a_0 \geq \delta(\GGamma)$ 
(which is needed if $k_1 = 0$ and $\eps_0(\Gamma_1) = 1$)

$\bullet$ $a_0 + \cdots + a_M \geq \delta(\GGamma)$ (which is needed if $k_m = M - l(\Gamma_m) + 1$ and $\eps_1(\Gamma_m) = 1$).

%
\noindent We need a different argument now since $\aa$ is no longer non-negative.

Let $(\alpha_0, \ldots, \alpha_M)$ be the divergence sequence for $\Delta(\cc,\dd)$ corresponding to the ``natural" orders of $D_l$ and $D_r$: weakly decreasing for $D_r$ and weakly increasing for $D_l$. The sequence of partial sums $\alpha_0+\cdots+\alpha_m$ is unimodal. Therefore, for $1 \leq m \leq M-1$,
\begin{eqnarray*}
\alpha_0+\cdots+\alpha_m &\geq& \min\{\alpha_0+\alpha_1, \alpha_0+\cdots+\alpha_{M-1}\} \\
&=& \min \{d^{\,t}+c^r_1-c^l_1, d^{\,b}+c^r_n-c^l_m\}
\geq 2\delta.
\end{eqnarray*}
Now observe that the difference
$(\alpha_0+\cdots+\alpha_m)-(a_0+\cdots+a_m)$ is naturally a sum of
terms $(r_j-r_i)$ (where $(i,j)$ is a reversal of $\rr$) and
$(l_i-l_j)$ (where $(i,j)$ is a reversal of $-\ll$). Therefore
$(\alpha_0+\cdots+\alpha_m)-(a_0+\cdots+a_m) \leq \delta(\ll,\rr)$. It
follows that, for $1 \le m \le M-1$,
\[
a_0+\cdots+a_m \geq (\alpha_0+\cdots+\alpha_m)-\delta(\ll,\rr) \geq 2\delta - \delta(\ll,\rr) = \delta+\delta(\GGamma) \geq 2\delta(\GGamma).
\]
proving the first series of inequalities.

The second and third inequality follow from our assumptions since $a_0 = d^{\, t}$ and   $a_0 + \cdots + a_M = d^{\,b}$.
%

\medskip
\noindent \emph{3. Piecewise polynomiality.}
The results of Section \ref{sec:discint} hold in exactly the same way for $\aa = \rr - \ll = \pi^r(\cc^r, \dd^r) - \pi^l(\cc^l, \dd^l)$. The only difference is that, as in Step 1 above, the domains of polynomiality now are given by how far $k_s$ is from  the numbers $d^{\,r}_1+\cdots+d^{\,r}_t$ (as before) and $d^{\,l}_1+\cdots+d^{\,l}_t$.

\medskip
\noindent \emph{4. Discrete integrals over polytopes.} Section \ref{sec:discint} holds without any changes.

\medskip
\noindent \emph{5. Polynomiality of Severi degrees.}
To prove Theorem \ref{thm:fixed} for general $h$-transversal polygons, the only adjustment we have to make is in the argument for quasipolynomiality (Step 2 of that proof). In this context, the $k_s$s can be anchored near the numbers $d^{\,r}_1+\cdots+d^{\,r}_t$ and $d^{\,l}_1+\cdots+d^{\,l}_t$, and we have to ensure that these anchor points are sufficiently far from each other. The exact same argument works if we assume that $|(d^{\,r}_1+\ldots+d^{\,r}_i) - (d^{\,l}_1+\ldots+d^{\,l}_j)| \geq \delta+2$. We now need to impose a bound of $\delta +2$ instead of $\delta +1$ because we apply the inequality $l(\pi) \leq \textrm{inv}(\pi)+1$ twice: for $\pi=\pi^r$ and $\pi = \pi^l$.
\end{proof}

\begin{proof}[Proof of Theorem \ref{mainthm:universal}]
In Step 1 above, notice that if $\cc$ is such that $c^r_i-c^r_{i+1} > \delta$ and $c^r_j-c^r_{j+1} > \delta$, then $\pi^r$ and $\pi^l$ must be empty. Therefore the proof of Theorem \ref{thm:universal} applies here as well. 
\end{proof}

\begin{remark}
\label{rmk:fewerhorizontaledges}
So far, our polynomiality results on Severi degrees 
are stated only for toric surfaces arising from
polygons with two sufficiently long horizontal edges, due to  the
assumptions $d^{\,t} \ge \delta$ and $d^{\,b} \ge \delta$. In particular,
this excludes the surface $\PP^2$.

By a slight modification of our argument one can show that there exist
universal polynomials for the families of Severi varieties of toric
surfaces associated to lattice
polygons with only one or no horizontal edge. More precisely, by
setting one or both of the numbers $d^{\,t}$ and $d^{\,b}$ to $0$, one can show a
universal polynomiality theorem analogous to Theorem~\ref{mainthm:universal},
with the conditions $d^{\,t} \ge \delta$ and/or $d^{\,b} \ge \delta$ removed when
appropriate. A proof of this variation can be obtained from our argument
by, in essence, disregarding the terms $\eps_0(\Gamma_1)$ and/or
$\eps_1(\Gamma_m)$ in the definition of the space $A(\GGamma, \aa)$ of possible
locations of the templates in a collection $\GGamma$.

A priori, the universal polynomials in these alternate settings are
different from the polynomials $p_\delta(\cc, \dd)$ of
Theorem~\ref{mainthm:universal}. However, we expect that they 
should be closely related; their relationship should be further clarified.
\end{remark}

\section{Explicit computations}
\label{sec:computations}

\subsection{Hirzebruch Surfaces}

Our results specialize as follows: For a non-negative integer $m$ let $F_m$ be the Hirzebruch surface
associated to the convex hull of $(0,0)$, $(0,1)$, $(1,1)$
and $(m+1, 0)$. In particular, $F_0 = \PP^1 \times \PP^1$. Let 
$N^{(F_m, \L_m(a,b)), \delta}$
be the degree of the
Severi variety of $\delta$-nodal curves in $F_m$ with bi-degree $(a,
b)$, i.e., of $\delta$-nodal curves whose Newton polygon is the convex hull of the
points $(0,0)$, $(0,b)$, $(a, b)$ and $(a+ bm, 0)$.

\begin{corollary} (Polynomiality of Severi degrees for Hirzebruch Surfaces.)
\noindent For fixed \newline $\delta \ge 0$, there exists a universal polynomial $p_\delta(a, b, m)$ such that the Severi degrees of Hirzebruch surfaces are given by
\begin{displaymath}
N^{(F_m, \L_m(a,b)), \delta} = p_\delta(a, b, m)
\end{displaymath}
for all positive integers $a, b, m$ with $a \geq \delta$, $a + m \ge 2 \delta $ and $b \ge \delta+1 $.

\label{cor:hirzebruch}
\end{corollary}
\begin{proof}
Following the proof of Theorem \ref{thm:maintheorem} we find that in this case there are no swaps. Therefore the proof of Theorem \ref{thm:maintheoremuniversal} applies to this case.
\end{proof}

\begin{remark}
\label{rmk:implicitHirzebruch}
The universal polynomials $p_\delta(a, b, m)$ for 
Hirzebruch surfaces $F_m$ for $\delta \le 3$ are: 
\begin{displaymath}
\tiny
\begin{split}
p_0(a,b,m) &=  1,\\
p_1(a,b,m) &=  3 b^2 m + 6 a b - 2 b m  - 4 a - 4 b + 4, \\
p_2(a,b,m) &= \frac{9}{2}  b^4 m^2
+ 18 a b^3 m  - 6  b^3 m^2 
+ 18 a^2 b^2  - 24 a b^2 m  - 12  b^3 m + 2 b^2 m^2
- 24 a^2 b - 24 a b^2 + 8 a b m  - b^2 m
+ 8 a^2 \\
&- 2  a b + 8 b^2   + \frac{23}{2} b m
+ 23 a  + 23 b 
- 30,\\
p_3(a,b,m) &=  \frac{9}{2}  b^6 m^3
 + 27 a b^5 m^2 - 9 b^5 m^3 
 + 54 a^2 b^4 m  - 54 a b^4 m^2 - 18  b^5 m^2   + 6  b^4 m^3 
+ 36 a^3 b^3  - 108 a^2  b^3 m - 72 a b^4 m \\
& + 36 a b^3 m^2  - 21 b^4 m^2 - \frac{4}{3} b^3 m^3 
 - 72 a^3 b^2  - 72 a^2 b^3 + 72 a^2 b^2 m  - 84 a b^3 m - 8 a  b^2
 m^2  + 24 b^4 m + \frac{137}{2} b^3 m^2 \\ 
& + 48 a^3 b  - 84 a^2 b^2  - 16 a^2 b m  + 48 a b^3 + 274 a b^2 m  + 137 b^3 m  - 31 b^2 m^2 
 - \frac{32}{3} a^3 + 274 a^2 b + 274 a b^2  - 124 a b m \\
&  - \frac{32}{3} b^3  - 68 b^2 m
 - 124 a^2  - 136 a b - 124 b^2 - \frac{374}{3} b m 
 - \frac{748}{3} a - \frac{748}{3} b
+ 452.
\end{split}
\end{displaymath}
Implicitly, 
for $0 \le \delta \le 5$,
the polynomials $p_\delta(a, b, m)$ 
 are given by
\[
\sum_{\delta \ge 0} p_\delta(a, b, m) x^\delta = \exp \Big(
\sum_{\delta \ge 1} q_\delta(a, b, m) x^\delta \Big),
\]
where
\begin{displaymath}
\begin{split}
q_1(a,b,m) &=3 (b^2 m + 2 a b ) - 2 ( b m  + 2 a + 2 b ) + 4, \\
q_2(a,b,m) &= \tfrac{1}{2}(- 42 (b^2 m +2 a b ) + 39 (b m + 2 a + 2 b ) - 76), \\
q_3(a,b,m) &= \tfrac{1}{3}(690 ( b^2 m + 2 a b ) - 788 ( b m + 2 a + 2 b ) + 1780), \\
q_4(a,b,m) &= \tfrac{1}{4}( - 12060 ( b^2 m  + 2 a b ) + 15945 ( b m + 2 a + 2 b ) - 41048), \\
q_5(a,b,m) &= \tfrac{1}{5}(217728 ( b^2 m + 2 a b ) - 321882 ( b m + 2 a + 2 b ) + 921864). \\
\end{split}
\end{displaymath}
These were computed by
a Maple implementation of the algorithm of 
Remark~\ref{rmk:algorithm}.
\end{remark}

\begin{remark}
\label{rmk:HirzebruchviaGottsche}
An alternative way to compute the polynomials $p_\delta(a, b, m)$ for
small~$\delta$ is to use the G\"ottsche-Yau-Zaslow
formula~\cite[Conjecture~2.4]{Go}, recently proved by
Tzeng~\cite{Tz10}. This formula states that there exist universal
power series $B_1(q)$ and $B_2(q)$  such that the Severi degrees
$N^{\delta}(S,\L)$ (i.e., the number of $\delta$-nodal curves in
$|\L|$ through an appropriate number of general points) of any smooth surface $S$ and sufficiently ample line bundle $\L$ are given by the generating function
\begin{equation}
\label{eqn:GYZ}
\sum_{\delta \geq 0}
N^{\delta}(S,\L)
(DG_2(\tau))^\delta=
\frac{(DG_2(\tau)/q)^{\chi(\L)}B_1(q)^{K_S^2}B_2(q)^{\L
    K_S}}{(\Delta(\tau)D^2G_2(\tau)/q^2)^{\chi({\mathcal O}_S)/2}},
\end{equation}
where $q=e^{2\pi i \tau}$, $G_2(\tau)=-\frac{1}{24}+ \sum_{n>0}\left(\sum_{d|n} d\right)q^n$ denotes the second Eisenstein series, $D=q\frac{d}{d\,q}$, $\Delta(\tau)=q\prod_{k>0}(1-q^k)^{24}$ is the Weierstrass $\Delta$-function,  and ${\mathcal O}_S$ is the structure sheaf of $S$. The formulas $\chi({\mathcal O}_S) = \frac1{12}(K_S^2+C_2(S))$ and $\chi(\L) = \frac12(\L^2-\L K_S) +  \frac1{12}(K_S^2+C_2(S))$ put everything in terms of the four numbers $\L^2, \L K_S, K_S^2$, and $C_2(S)$.
 
The formula above allows us to compute the
polynomials $q_\delta(a,b,m)$ from the Chern classes $C_2(T_{F_m})$, $C_1(\L_m(a,b))$,
$C_1(K_{F_m})$ for the Hirzebruch surface $F_m$ and the line bundle $\L_m(a,b)$
determined by $a$ and $b$, together with the coefficients of $B_1$ and
$B_2$ (if these are known). More specifically, the
first $t$ coefficients of $B_1$ and $B_2$ determine the polynomials
$q_\delta(a,b,m)$ for $\delta \le t$ (and vice versa) for any $t \ge 1$. The second
author rigorously established the first $14$ coefficients of $B_1$ and
$B_2$, by computing the node polynomials for $\PP^2$ for $\delta \le
14$. This extended work of Kleiman and Piene~\cite{KP} for $\delta
\le 8$ and confirmed the prediction of G\"ottsche~\cite{Go}. Using
this method, one can in principle compute the polynomials $p_\delta(a,
b, m)$ and $q_\delta(a,b,m)$ for $\delta \le 14$. We note, however,
that the methods of this paper to compute $p_\delta(a, b, m)$ are less
efficient than in the $\PP^2$ case~\cite{FB}. With the current computational limitations, we expect computability of $p_\delta(a, b, m)$ in feasible time only for $\delta
\le 7$ or $8$.
\end{remark}


\begin{remark}
As Hirzebruch surfaces $F_m$ are smooth for all $m \ge 0$ and the four numbers 
$\L^2, \L K_S, K_S^2$, and $C_2(T_{F_m})$ are polynomial in $a, b$, and $m$ in this case, the polynomiality result of
Corollary~\ref{cor:hirzebruch} also follows (with a weaker threshold)
from Tzeng's proof~\cite[Theorem~1.1]{Tz10} of G\"ottsche's Conjecture. 
\end{remark}

\subsection{A non-smooth example}
\label{sec:non-smooth-example}

We now compute the node polynomials of
a family of singular toric surfaces  for $\delta = 1$ and $\delta = 2$. For positive integers $c, d_0, d_1$ and $d_2$,
let $\Delta(c; d_0, d_1, d_2)$ be the convex
hull of the points $(0,0)$, $ (0, d_1+d_2)$, $(d_0, d_1 + d_2)$, $(d_0 +
d_1c, d_2)$ and $(d_0 + d_1 c, 0)$; see
Figure~\ref{fig:non-smooth_polygon}. The corresponding toric surface
$S(c)$ is singular whenever $c \ge 2$. The Severi degree
$N^{(d_0, d_1, d_2),\delta}_{S(c)}$ counts $\delta$-nodal curves whose
Newton polygon is $\Delta(c; d_0, d_1, d_2)$.

\begin{figure}[h]
\begin{center}
\qquad \includegraphics[height=3cm]{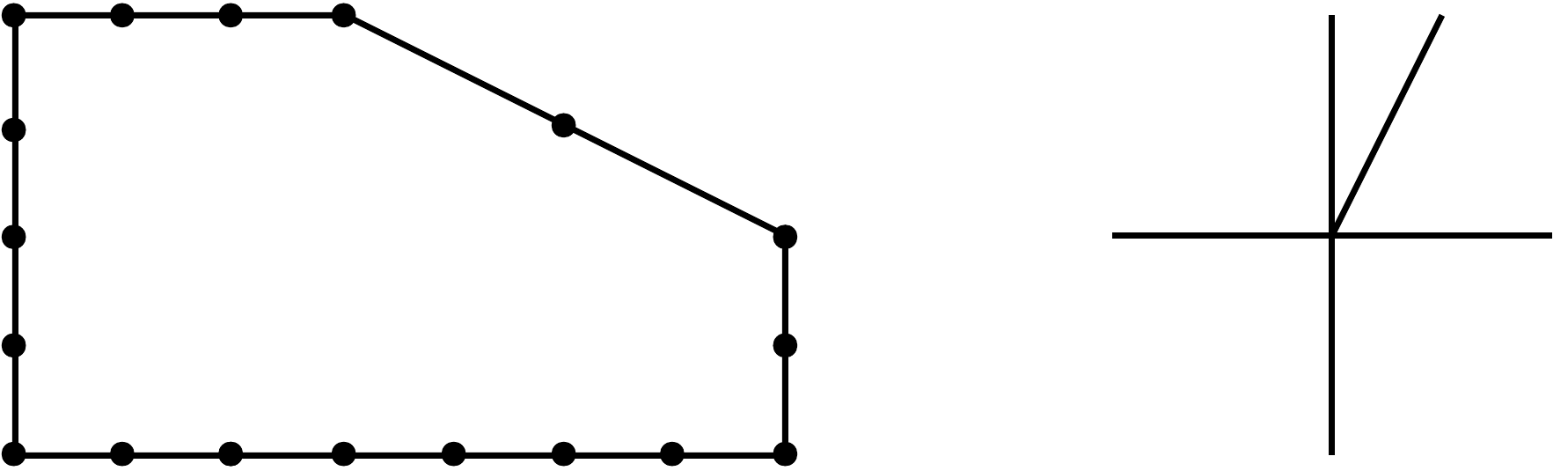}
\end{center}
\caption{The convex hull of  $(0,0), (0, d_1+d_2), (d_0, d_1 + d_2), \break (d_0 +
d_1c, d_2)$, and $(d_0 + d_1 c, 0))$, together with its normal fan. The cone dual to
the vertex $(d_0 + d_1c, d_2)$ corresponds to a singular point of
$S(c)$ if $c \ge 2$.}
\label{fig:non-smooth_polygon}
\end{figure}

\begin{corollary}[Polynomiality of Severi degrees for a non-smooth
  surface]
For fixed $\delta \ge 0$, there exists a universal polynomial
$p_\delta(d_0, d_1, d_2, c)$ such that
\[
N^{(d_0, d_1, d_2),\delta}_{S(c)} = p_\delta(d_0, d_1, d_2, c)
\]
for all positive integers $d_0, d_1, d_2, c$ with $d_0 \ge \delta$, $d_0
+ c \ge 2 \delta$, $d_1, d_2 \ge \delta$, $c \ge \delta + 1$. 
\end{corollary}

\begin{proof}
Since $c \ge \delta + 1$, there are no swaps and the  proof of
Theorem~\ref{thm:universal} applies.
\end{proof}

\begin{remark}
\label{rmk:computingsingularcase}
Using 
the algorithm of 
Remark~\ref{rmk:algorithm}, we find that the universal polynomials $p_\delta(d_0, d_1, d_2, c)$ for $\delta \le
2$ are:
\begin{displaymath}
\tiny
\begin{split}
p_0(d_0, d_1, d_2, c) = & 1,\\
p_1(d_0, d_1, d_2, c) = & 3 (d_1^2 c+2 d_1 d_2 c+2 d_0 d_1+2 d_0 d_2) -
2 (d_1 c + 2 d_0 + 2 d_1 + 2 d_2) + 4, \\
p_2(d_0, d_1, d_2, c) = & \tfrac{9}{2} d_1^4 c^2+18 d_1^3 d_2 c^2+18
d_1^2 d_2^2 c^2+18 d_0 d_1^3 c+54 d_0 d_1^2 d_2 c
 +36 d_0 d_1 d_2^2 c-6 d_1^3 c^2-12 d_1^2 d_2 c^2+18 d_0^2 d_1^2 \\
& + 36 d_0^2 d_1 d_2+18 d_0^2 d_2^2-24 d_0 d_1^2 c-36 d_0 d_1 d_2 c-12
d_1^3 c-36 d_1^2 d_2 c+2 d_1^2 c^2-24 d_1 d_2^2 c+2 d_1 d_2 c^2 \\
& -24 d_0^2 d_1-24 d_0^2 d_2-24 d_0 d_1^2-48 d_0 d_1 d_2+8 d_0 d_1
c-24 d_0 d_2^2+2 d_0 d_2 c-d_1^2 c-10 d_1 d_2 c-2 d_1 c^2 \\
& +8 d_0^2-2 d_0 d_1-2 d_0 d_2-2 d_0 c+8 d_1^2+16 d_1 d_2+\tfrac{23}{2} d_1 c+8 d_2^2+23 d_0+23 d_1+23 d_2+4 c-30.
\end{split}
\end{displaymath}
Equivalently, the polynomials $p_\delta(d_0, d_1, d_2, c)$, for
$\delta \le 2$,
are given by
\[
\sum_{\delta \ge 0} p_\delta(d_0, d_1, d_2, c) x^\delta = \exp \Big(
\sum_{\delta \ge 1} q_\delta(d_0, d_1, d_2, c) x^\delta \Big),
\]
where
\begin{displaymath}
\begin{split}
q_1(d_0, d_1, d_2, c) = & 3 (d_1^2 c+2 d_1 d_2 c+2 d_0 d_1+2 d_0 d_2) +
2 (- d_1 c - 2 d_0 - 2 d_1 - 2 d_2) + 4, \\
q_2(d_0, d_1, d_2, c) = & - 42( d_1^2 c
+ 2 d_1 d_2 c
+ 2 d_0 d_1
+ 2 d_0 d_2)
- 39 (-d_1 c
- 2 d_0
- 2 d_1
- 2 d_2) \\
& 
+ 4(d_1 c + d_0)(d_2 - 1)c
+ 8 c-76.
\end{split}
\end{displaymath}

Let $T_\delta(w, x,y,z)$ be G\"ottsche's universal polynomials for the smooth case  (c.f.\, Section~\ref{sec:relationwithGoettscheConjecture}), and define
polynomials $Q_\delta(w,x,y,z)$ via
$$\sum_{\delta \ge 1} T_\delta(w, x,y,z) t^\delta = \exp \left(
  \sum_{\delta \ge 1} Q_\delta(w,x,y,z) t^\delta \right).$$
According to the G\"ottsche-Yau-Zaslow
formula~(\ref{eqn:GYZ}), the polynomials
$Q_\delta(w, x,y,z)$ satisfy, for $\delta \le 2$,
\[
\begin{split}
Q_1(\L^2, \L K_S, K_S^2, c_2(S))  =& \, 3 \L^2 + 2 \L K_S +  c_2(S), \\
Q_2(\L^2, \L K_S, K_S^2, c_2(S))  = & - 42 \L^2 - 39 \L K_S - 6 K_S^2 -
7 c_2(S). \\
\end{split}
\]
If G\"ottsche's conjecture held in this non-smooth example, we would
have $Q_1=q_1$ and $Q_2=q_2$. For our example, we have \begin{displaymath}
\begin{split}
\L^2 & = d_1^2 c+2 d_1 d_2 c+2 d_0 d_1+2 d_0 d_2,\\
\L \cdot K_{S} & = -(d_1 c + 2 d_0 + 2 d_1 + 2 d_2),\\
c_2^{\text{MP}} &= 5, \\
K_{S}^2 &= 8 - c
\end{split}
\end{displaymath}
where the first two computations are in the singular cohomology of the
toric variety $S(c)$ (\cite[Theorem~12.4.1]{CLS11}), and $c_2^{MP}$ is MacPherson's Chern class as computed in~\cite{BBF92}. Then we get
\[
\begin{split}
Q_1=&  \, 3 (d_1^2 c+2 d_1 d_2 c+2 d_0 d_1+2 d_0 d_2) 
+  2 (- d_1 c - 2 d_0 - 2 d_1 - 2 d_2) + 5, \\
Q_2 =&   - 42( d_1^2 c
+ 2 d_1 d_2 c
+ 2 d_0 d_1
+ 2 d_0 d_2)
-   39 (-d_1 c
- 2 d_0
- 2 d_1
- 2 d_2) +6c-83.
\end{split}
\]

These expressions for $Q_1$ and $Q_2$ bear some similarity with the correct expressions for $q_1$ and $q_2$ above, but they do not coincide; so G\"ottsche's formula for the smooth case does not apply to this surface. However, this example seems to suggest that some modification of G\"ottsche's formula should still apply to a more general family of surfaces. We do not know what that modification would look like.

%
%
\end{remark}

\section{Further directions and open problems}

Our work suggests several directions of further research, some of which we have alluded to throughout the paper. We collect them here.

\begin{itemize}
\item 
As mentioned in Section~\ref{sec:intro} the relationship between our work and G\"ottsche's Conjecture needs to be further clarified. G\"ottsche's Conjecture is stated for smooth surfaces, while the surfaces we consider are generally not smooth. Is there a common generalization?
\item
We suspect that Severi degrees of \textbf{any} large toric surface are universally polynomial, even though we have only been able to prove it for large $h$-transverse toric surfaces. This restriction comes from Brugall\'e and Mikhalkin's observation that the encoding of tropical curves into floor diagrams only works in the $h$-transverse case. Can we adjust the definition of a floor diagram, or find a different combinatorial encoding that allows us to drop this restriction? This could involve making a different choice for our generic collection of points of Section \ref{sec:tropicalmethod}.
\item
It should be possible to weaken the conditions on $\cc$ and $\dd$ in Theorems \ref{mainthm:fixed} and Theorem \ref{mainthm:universal}. It is possible that we can drop the conditions on $\cc$ entirely; some conditions on $\dd$ are surely necessary.
\item
It would be of interest, and probably within reach, to clarify how the polynomials $p_{\delta}(\cc,\dd)$ vary when we drop horizontal edges from $\Delta$, or when we vary the lengths $m$ and $n$ of their input.
\end{itemize}

\bibliographystyle{amsplain}
\bibliography{References_Florian}

\end{document}